\documentclass[conference]{IEEEtran}
\IEEEoverridecommandlockouts
\usepackage{macros}
\usepackage{cite}
\usepackage{todonotes}
\usepackage{amsmath,amssymb,amsthm,amsfonts}
\usepackage{graphicx}
\usepackage{textcomp}
\usepackage{hyperref}

\usepackage{tabularray}   
\usepackage{adjustbox}

\newtheorem{theorem}{\textbf{Theorem}}[section]
\newtheorem{lemma}{\textbf{Lemma}}[section]
\newtheorem{proposition}{\textbf{Proposition}}[section]

\newtheorem{assumption}{\textbf{Assumption}}[section]
\newtheorem{remark}{\textbf{Remark}}[section]
\newtheorem{definition}{\textbf{Definition}}[section]
\renewcommand{\IEEEQED}{\IEEEQEDopen}
\makeatletter
\def\theorem@qed{\pushQED{\qed}\qedhere\popQED}
\makeatother

\usepackage{xcolor}
\def\BibTeX{{\rm B\kern-.05em{\sc i\kern-.025em b}\kern-.08em
    T\kern-.1667em\lower.7ex\hbox{E}\kern-.125emX}}
\begin{document}

\title{\algoname: A compressed Newton-type fully distributed optimization algorithm*
}

\author{\IEEEauthorblockN{Souvik Das}
\IEEEauthorblockA{\textit{Dept. of Electrical Eng.} \\
\textit{Uppsala University}\\
Uppsala, Sweden \\
souvik.das@angstrom.uu.se}
\and
\IEEEauthorblockN{Subhrakanti Dey}
\IEEEauthorblockA{\textit{Dept. of Electrical Eng.} \\
\textit{Uppsala University}\\
Uppsala, Sweden \\
subhrakanti.dey@angstrom.uu.se}
}

\maketitle
\begin{abstract} Compression techniques are essential in distributed optimization and learning algorithms with high-dimensional model parameters, particularly in scenarios with tight communication constraints such as limited bandwidth. This article presents a \emph{communication-efficient second-order} distributed optimization algorithm, termed as \(\algoname\), equipped with a compression module, designed to minimize the average of local strongly convex functions. \(\algoname\) incorporates two consensus-based averaging steps at each node: gradient tracking and approximate Newton-type iterations, inspired by the recently proposed  Network-GIANT. Under certain sufficient conditions on the step-size, \(\algoname\) achieves significantly faster linear convergence, comparable to that of its first-order counterparts, both in the compressed and uncompressed settings. \(\algoname\) is efficient in terms of data usage, communication cost, and run-time, making it a suitable choice for distributed optimization over a wide range of wireless networks. Extensive experiments on synthetic data and the widely used CovType dataset demonstrate its superior performance.
\end{abstract}


\section{Introduction}
\label{sec:intro}
Data compression techniques such as quantization and sparsification have gained popularity in decentralized optimization in recent years, driven by the growing number of handheld smart devices capable of performing complex machine learning tasks remotely. These devices function as nodes of multi-agent systems \cite{ref:DKM-FD-HS-SHL-SC-RB-JL-17,ref:AN-JL-18}, exchanging data with centralized servers and/or among themselves. Such compression techniques enable communication-efficient large-scale optimization on local devices by reducing the amount of data to be exchanged among themselves while maintaining performance and achieving convergence rates comparable to their non-compressed counterparts, particularly in scenarios with tight link bandwidth constraints between individual nodes.

While most research on compressed fully decentralized algorithms focuses on first-order gradient-descent–based approaches \cite{ref:AK-TL-SUS-MJ-19, ref:XL-YL-RW-JT-MY-20, ref:YL-ZL-KH-SP-22, ref:JZ-KY-LX-23}, the literature on fully-distributed communication-efficient \emph{approximate} Newton-type algorithms is sparse, with a few exceptions. The authors in \cite{ref:HL-JZ-AMCS-QL-23} proposed a Newton-type distributed optimization algorithm that enhances communication efficiency by compressing the Hessian. While such a compression scheme is a natural choice, it often results in a larger amount of bit transmission than innovation compression schemes developed for gradient-descent-based algorithms such as \(\cgt\) \cite{ref:YL-ZL-KH-SP-22} or \(\cold\) \cite{ref:JZ-KY-LX-23}. Our proposed algorithm, \(\algoname\), employs innovation compression in conjunction with an approximate Newton-type algorithm for objective functions satisfying strong convexity and smoothness conditions. \(\algoname\) is a fully distributed approximate Newton-type algorithm. It simultaneously exploits the curvature of the objective function to achieve significantly faster convergence rates than the first-order methods, handles general compression schemes with arbitrary precision, achieves the communication efficiency of the same order of gradient-based methods and attains comparable convergence rates to its uncompressed counterpart.

The main contributions of the article are as follows:
\begin{itemize}[leftmargin = *, label = \(\circ\)]
    \item We present \(\algoname\), a fully distributed Newton-type optimization algorithm. It extends the recently developed Network-GIANT \cite{ref:AM-GS-LS-SD-23} by integrating a data compression module with a gradient-tracking–based approximate Newton update. \(\algoname\) supports a broad class of compression schemes with arbitrary precision.

    \item We establish that the \(\algoname\) achieves a linear convergence rate comparable to that of the uncompressed Network-GIANT (with \(\agents\) agents and a \(\dimdv\)-dimensional model) with arbitrary initialization,  with a  per-iteration communication overhead of \(\Bigoh{( \agents \dimdv)}\), equivalent to the first-order gradient-based optimization methods, while achieving a significantly faster convergence rate. under a set of sufficient conditions, \(\algoname\) exhibits linear convergence; see Lemma \ref{lem: key lemma} and Theorem \ref{thm:key thm}.

    \item The empirical evidence from experiments on synthetic data and the widely used CovType dataset consistently demonstrates  (see \S~\ref{sec:experiments}) superior convergence rates and communication efficiency across  different classes of compression schemes --- including biased, unbiased, and general compression (see Definition \ref{def:General Compression Operator}) --- when compared with \(\cgt\) \cite{ref:YL-ZL-KH-SP-22}, \(\lead\) \cite{ref:XL-YL-RW-JT-MY-20}, \(\cold\) \cite{ref:JZ-KY-LX-23}, and the algorithm in \cite{ref:HL-JZ-AMCS-QL-23}. 
\end{itemize}

\subsection*{Organization}The rest of the paper is organized  as follows: \S~\ref{sec:ps} describes the problem and states the key assumptions. In \S~\ref{sec:algo} we present the main steps involved in \(\algoname\), followed by its convergence analysis. The numerical experiments are included in \S~\ref{sec:experiments}. The proofs of the main results are relegated to Appendix \ref{sec:proofs}.
\section{Problem formulation and notations}
\label{sec:ps}
\subsubsection*{Notations}
We employ the following notations throughout the article. For any square matrix \(D \in \Rbb^{\agents \times \agents}\), the  \(\spec(D)\) refers to the spectral radius of \(D\). We denote by \(\norm{D} = \sup_{x \neq 0, x \in \Rbb^n} \frac{\norm{Ax}}{\norm{x}}\) the matrix norm. All vector norms are assumed to be Euclidean unless specified otherwise.

\subsection{Problem description}\label{subsec:ps}
Consider the distributed optimization problem 
\begin{equation}\label{eq:key prob}
    \min_{\dvar \in \Rbb^{\dimdv}} \objfunc(\dvar) \Let \frac{1}{\agents}\sum_{i=1}^{\agents} \objfunc_i(\dvar),
\end{equation}
with \(\agents\)-agents connected over a network \(\network\). Here \(\dvar \in \Rbb^{\dimdv}\) is the decision variable, for each \(i\)-th agents where \(i \in \aset[]{1, 2, \ldots, \agents}\) the function \(\objfunc_i: \Rbb^{\dimdv} \mapsto \Rbb\) denotes the local objective function, which depends on the local data \(\ldata_i\). The sequence \(\ldata \Let \bigl(\ldata_i \bigr)_{i=1}^{\agents}\) represents available to the agents, and in general, it can be heterogeneous.  The objective of the distributed optimization problem is for all nodes to find the global minimizer through local updates and communications with their neighbors only.


The agents are connected through a communication network, modeled as an \emph{undirected} graph \(\network = \aset[\big]{\nodeset, \edgeset}\), where \(\nodeset = \aset[]{1,2,\cdots,\agents}\) denotes the set of nodes (or vertices) representing the agents, and \(\edgeset \subset \nodeset \times \nodeset\) denotes the set of edges. Each edge \((i,j) \in \edgeset\) represents an undirected link between agents \(i\) and \(j\), allowing them to communicate and exchange data directly. For each agent \(i \in \nodeset\), we define the set of \emph{neighbors} by \(\ngbr_i \Let \aset[]{j \in \nodeset \suchthat (i,j) \in \edgeset}\). 

\subsection{Standing assumptions}\label{subsec:sa}
\begin{assumption}[Consensus weight matrix]
    \label{assum:on graph}
    A nonnegative doubly stochastic consensus weight matrix \(\Wght = [\wght_{ij}]_{i,j=1}^{\agents}\), satisfying \(\sum_{i'} \wght_{i'j} = \sum_{j'} \wght_{ij'}\) for every \(i',j' \in \nodeset\), encodes the network structure. Agent \(i\) can receives data from agent \(j\) if and only if \(\wght_{ij}>0\). Moreover, for \(i \in \nodeset\) we have \(\wght_{ii}>0\).
\end{assumption}

\begin{assumption}[Regularity of local objective functions]
    \label{assum:Standard assumptions}
    We impose the standing assumptions: for each \(i \in \aset[]{1,2,\ldots,\agents}\), \(\objfunc_i: \Rbb^{\dimdv} \mapsto \Rbb\) is twice differentiable, \(\strconv\)-strongly convex, and admits \(\lips\)-Lipschitz gradients, i.e., there exist positive constants \(\strconv,\ \lips\) such that for each \(i \in \aset[]{1,2,\ldots, \agents}\) and for all \(\dvar,y \in  \Rbb^{\dimdv}\), we get
    \begin{equation}\label{eq:con_smoothness}
        \begin{cases}
            \hess \objfunc_i (\dvar) \succeq \strconv \identity{\dimdv}\\
            \norm{\grd \objfunc_i(\dvar) - \grd \objfunc_i(y)} \leq \lips \norm{\dvar - y}.
        \end{cases}
    \end{equation}
\end{assumption}
The above two conditions in \eqref{eq:con_smoothness} imply that 
\begin{equation}
    \label{eq:strongly_convex_Lips}
    \strconv \identity{\dimdv} \preceq \hess \objfunc_i(\dvar) \preceq \lips \identity{\dimdv} \text{ for every }\dvar \in \Rbb^{\dimdv},
\end{equation}
for each \(i \in \aset[]{1,2,\ldots, \agents}\). It also follows immediately that the global objective function $f(x)$ is strongly convex and Lipschitz continuous with the same parameters and $\mu \identity{\dimdv} \preceq \nabla^2 f(x) \preceq \lips \identity{\dimdv}$. 
Note that the optimization problem \eqref{eq:key prob} admits an optimal solution \(\dvar^{\ast}\), which follows as an immediate consequence of strong convexity \cite[p. no. \(460\)]{ref:SB-LV-04}. 

\subsection{Compression operators}\label{subsec:compression schemes}
\begin{definition}{\cite[Assumption \(5\)]{ref:YL-ZL-KH-SP-22}}
    \label{def:General Compression Operator}    The operator \(\gcompop: \Rbb^{\dimdv} \lra \Rbb^{\dimdv}\) is a \emph{general compression operator}, if there exists a constant \(\Cgen>0\) such that
    \begin{equation}
        \label{eq:general compressor operator}
        \EE \expecof[\big]{\norm{\gcompop(\dvar) - \dvar}^2} \le \Cgen \norm{\dvar}^2 \quad \text{for all }\dvar \in \Rbb^{\dimdv},
    \end{equation}
    where the expectation \(\EE \expecof[]{\cdot} \equiv \EE_{\gcompop} \expecof[]{\cdot}\) is applied with respect to the underlying randomness associated with the compression operator.\footnote{Throughout we suppress \(\gcompop\) in the expectation for the ease of writing 
    notation.} The \(\scale\)-scaling of the operator \(\gcompop\), where the scaling parameter \(\scale > 0\), satisfies
    \begin{equation*}
        \EE \expecof[\Bigg]{\norm{\frac{\gcompop(\dvar)}{r} - \dvar}^2} \le (1-\compar) \norm{\dvar}^2 \quad \text{for all }\dvar \in \Rbb^{\dimdv},
    \end{equation*}
    for some constant \(\compar \in \lorc{0}{1}\).
\end{definition}
\begin{remark}
    \label{rem:on different operators}
    With suitable scaling, the unbiased and biased compression operators \cite{ref:AB-SH-PR-MS-23} can be recovered from the general compression operator; see \cite{ref:YL-ZL-KH-SP-22,ref:AB-SH-PR-MS-23} for details on these derivations. The convergence analysis is presented for a general compression scheme, and for completeness, we include experiments with unbiased and biased operators.
\end{remark}

\section{\(\algoname\) and its analysis}
\label{sec:algo}
\subsection{Key steps of \(\algoname\)}\label{subsec:algo}
\(\algoname\) extends the core idea of Network-GIANT established in \cite{ref:AM-GS-LS-SD-23} to include data compression. It is described by the following recursions: for each \(i\)-th agent, 
\begin{align}
    \label{eq:agent con step}
    &\dvar_i(t+1) = \dvar_i(t) - \constsz \big(\edv_i(t) - \wedv_i(t) \big) - \stsz \cdvar_i(t),\nn \\
    &\gtvar_i(t+1) = \gtvar_i(t) - \constsz \big( \egt_i(t) - \wegt_i(t)\big) + \gt_i(t+1)
\end{align}
for each iteration \(t \in \Nz\), where \(\constsz, \stsz> 0 \) are scalars to be picked appropriately, and \(\gtvar_i(0) = \grd \objfunc_i\big(\dvar_i(0)\big)\) must be satisfied. Here the variable \(\gtvar(t)\) for each \(i\), refers to as  a gradient tracker \cite{ref:GQ-NL-17}. For each \(i \in \aset[]{1,2,\ldots, \agents}\), the variable \(\cdvar_i(t) \Let \hess \objfunc_i(\dvar_i(t)){\inverse} \gtvar_i(t) \in \Rbb^{\dimdv}\) represents the local descent direction, and \(\gt_i(t) \Let \grd \objfunc_i\big(\dvar_i(t)\big) - \grd \objfunc_i\big(\dvar_i(t-1)\big)\) with \(\gt_i(0) =\grd\objfunc_i\big(\dvar_i(0)\big) \in \Rbb^{\dimdv}\).  The variables \((\edv_i(t),\egt_i(t))\) denote the estimates of \(\dvar(t)\) and \(\gtvar(t)\) for each \(t\) and each agent \(i\), to be defined later. Moreover, \((\wedv_i(t), \wegt_i(t))\) are the weighted average of estimates communicated by the one-hop neighbors, and defined by \(\wedv_i(t) = \sum_{j \in \ngbr_i}w_{ij} \edv_j(t)\) and \(\wegt_i(t) = \sum_{j \in \ngbr_i}w_{ij} \egt_j(t)\). Here \(\dvar_i(t) \in \Rbb^{\dimdv}\) and \(\gtvar_i(t) \in \Rbb^{\dimdv}\) denote the local copies \(\dvar_i(t) \in \Rbb^{\dimdv}\) and \(\gtvar_i(t) \in \Rbb^{\dimdv}\) of \(\dvar(t) \in \Rbb^{\dimdv}\) and \(\gtvar(t) \in \Rbb^{\dimdv}\)  at the \(t\)-th iteration, respectively. 

In what follows, we adopt a matrix notation to compactly represent the recursion \eqref{eq:agent con step}. To that end, we define the global decision variables  \(\dmat(t) \Let \bigl(\dvar_1(t)\; \cdots\; \dvar_{\agents}(t) \bigr)^{\top} \in \Rbb^{\agents \times \dimdv}\) and  \(\gtmat(t) \Let \bigl(\gtvar_1(t)\; \cdots \; \gtvar_{\agents}(t) \bigr)^{\top} \in \Rbb^{\agents \times \dimdv}\). The above defined notations carry forward also to the remaining global variables \(\cdmat(t) \in \Rbb^{\agents \times \dimdv}\), \(G(t) \in \Rbb^{\agents \times \dimdv}\), and \(\big(\edvm(t),\egtm(t)\big) \in \Rbb^{\agents \times \dimdv} \times \Rbb^{\agents \times \dimdv}\) and \(\big(\wedvm(t),\wegtm(t)\big) \in \Rbb^{\agents \times \dimdv} \times \Rbb^{\agents \times \dimdv}\). Let \(\Objfunc: \Rbb^{\agents \times \dimdv} \lra \Rbb\) be the aggregate objective function, defined for each iteration \(t\) by \(\Objfunc(\dmat(t)) \Let \sum_{i=1}^{\agents}\objfunc_i(\dvar_i(t))\), and let \(\grd \Objfunc(\dmat(t)) \Let \bigl(\grd \objfunc_1(\dvar_1(t)) \; \cdots \; \grd \objfunc_{\agents}(\dvar_{\agents}(t)) \bigr)^{\top} \in \Rbb^{\agents \times \dimdv}\). We now compactly re-write the consensus steps \eqref{eq:agent con step}: for \(t \in \Nz\),
\begin{align}\label{eq:update rule}
    &\dmat (t+1) = \dmat(t) - \constsz \bigl(\edvm(t) - \wedvm(t)\bigr) - \stsz \cdmat(t) \nn \\ 
    &\gtmat(t+1) = \gtmat(t) - \constsz \bigl(\egtm(t) - \wegtm(t)\bigr) + 
    \G(t+1),
\end{align}
with \(\gtmat(0) = \grd \Objfunc\big(\dmat(0)\big)\). The quantity \(\stsz>0\) is the step-size of \(\algoname\). Here \(\constsz>0\) is the consensus step-size associated with the underlying compression scheme. It controls the compression error feedback, denoted by \(\bigl(\dmat(t) - \edvm(t)\bigr)_{t \in \Nz}\), fed into the consensus steps \eqref{eq:update rule}.
Let us highlight the key features of \(\algoname\):
\subsubsection*{The \(\compress\) module}
This module, motivated by \cite{ref:XL-YL-RW-JT-MY-20,ref:YL-ZL-KH-SP-22}, incorporates the compression and the communication steps for every iteration across agents. We introduce the auxiliary states \((\auxdv,\auxgt)\), and instead of directly compressing \(\dmat\) and \(\gtmat\) or directly the Hessian as in \cite{ref:HL-JZ-AMCS-QL-23}, the differences \(\dmat - \auxdv\) and \(\gtmat - \auxgt\) are compressed. The compressed signals are denoted by \(Q_{\dvar}(t) \Let \compscheme \big(\dmat(t) - \auxdv(t) \big)\) and \(Q_{\gtvar}(t) \Let \compscheme \big(\gtmat(t) - \auxgt(t) \big)\), which are utilized to compute the estimates \(\edvm\) and \(\egtm\), respectively. Each agent then communicates these estimates to their neighboring agents. One can show that the compression error follows \(\EE \expecof[\bigg]{\norm{\dmat(t) - \edvm(t)}^2}
\le \EE \expecof[\Big]{\norm{\dmat(t) - \auxdv(t)}^2}\), 
and as \(\norm{\dmat(t) - \auxdv(t)} \xrightarrow[t \ra + \infty]{} 0\) the corresponding compression error also converges to zero asymptotically.

From Algorithm \ref{alg:sec_ord_comp}, if \(Q_z = \compscheme \bigl(Z-H \bigr) = Z\) is imposed (no compression), it results in \(\widehat{Z} = Q_z + H = Z\), and it is not difficult to show that \(\widehat{Z}^w = \Wght Z\). Substituting this in \eqref{eq:update rule}, we recover the original Network-GIANT algorithm 
with the consensus weight matrix \(\wt{\Wght} \Let (1-\constsz)\identity{\agents} + \constsz \Wght\).
\subsubsection*{Gradient tracking}
The second consensus step in \eqref{eq:update rule}, first established in \cite{ref:GQ-NL-17}, tracks the past gradients with \emph{zero error} asymptotically, and is critical for convergence to the exact optimal solution. Indeed, if \(\gtmat(0) = \grd \Objfunc\big(\dmat(0)\big)\) is imposed, one can show that \(\frac{1}{\agents} \ones^{\top}\gtmat(t) = \frac{1}{\agents}\sum_{i=1}^{\agents}\grd \objfunc_i\big(\dvar_i(t)\big)\) for each iterations \(t \in \Nz\), even under the compression scheme. For the proof, we refer to Lemma \ref{it:port_1} in Appendix \ref{subappen:technical}.

\subsubsection*{Second-order oracle}
Approximate Newton-like methods are reported to achieve faster convergence over the first-order gradient based algorithms. We adopt the crux of Network-GIANT, where the sequence of vectors \((\edv_i(t), \egt_i(t))_{i = 1}^{\agents}\) for each \(t\), are exchanged among the agents, instead of the local Hessians directly. This approach results in a per-iteration communication complexity of \(\Bigoh{\big( \agents \dimdv\big)}\) 
only.
Note also that in the convergence analysis, the regularity properties of local Hessian is extensively utilized.

To avoid notational clutter, we will use the symbols \(\dvtuple(t)\),  \(\gttuple(t)\), and \(\G(t+1)\) to represent \(\big(\dmat(t), \auxdv(t), \wauxdv(t) \big)\), \(\big(\gtmat(t), \auxgt(t), \wauxgt(t) \big)\), and \(\grd \Objfunc\bigl(\dmat(t+1)\bigr) - \grd \Objfunc\bigl(\dmat(t)\bigr)\) with \(\G(0) = \grd \Objfunc(0)\) for each \(t\), respectively, in Algorithm \ref{alg:sec_ord_comp}. The stopping time \(\horizon\) is chosen such that for \(t \ge \horizon\), \(\norm{\grd \Objfunc(\dmat(t))} \le c\) for a pre-specified tolerance \(c>0\).
\begin{algorithm}[!ht]
    \SetAlgoLined
    \DontPrintSemicolon
    \SetKwInOut{ini}{Initialize}
    \SetKwInOut{giv}{Data}
    \SetKwInOut{out}{Output}
    
    \giv{Stopping time \(\horizon\), step-size \(\stsz\), consensus step-size \(\constsz \in \lorc{0}{1}\), scaling parameters \((\condv,\congt)\),  consensus weight matrix \(\Wght\)}
    
    \ini{\(\dmat(0), \auxdv(0), \auxgt(0)\), impose \(\gtmat(0) = \grd \Objfunc(\dmat(0))\), 
     \(\wauxdv(0) = \Wght \auxdv(0)\), \(\wauxgt(0) = \Wght \auxgt(0)\)}
    
    \For{\(t = 0, 1, 2, \ldots, \horizon\)}{
        \texttt{Apply the compression scheme:} \label{line:compress_scheme}\\
        \(\begin{aligned}
            &\edvm(t), \wedvm(t), \auxdv(t+1), \wauxdv(t+1) = \compress\bigl(\dvtuple(t)\bigr) \\
            &\egtm(t), \wegtm(t), \auxgt(t+1), \wauxgt(t+1) = \compress\bigl(\gttuple(t)\bigr)
        \end{aligned}\)

        \texttt{Update rule:} \label{line:update_main}\\
        \(\begin{aligned}
            &\dmat(t+1) \gets \dmat(t) - \constsz \bigl(\edvm(t) - \wedvm(t)\bigr) - \stsz \cdmat(t) \\
            &\gtmat(t+1) \gets \gtmat(t) - \constsz \bigl(\egtm(t) - \wegtm(t)\bigr) + \G(t+1)
        \end{aligned}\)
    }
    
    \out{\(\bigl( \dmat(t), \gtmat(t) \bigr)_{t=0}^{\horizon}\)}
    
    \BlankLine
    
    \SetKwFunction{FCompress}{Compress}
    \SetKwProg{Fn}{Procedure}{:}{}
    \Fn{\FCompress{\(Z, H, H^{\wght}\)}}{
        \texttt{Encode: }\(Q_z = \compscheme \bigl(Z-H \bigr) \) \label{line:encode}
        
        \texttt{Update: }\(\widehat{Z} = Q_z + H\) \label{line:update}

        \texttt{Communication: }\(\widehat{Z}^{\wght} = H^{\wght} + \Wght Q_z \) \label{line:cummunicate}

        \texttt{Update:} \label{line:aux_update}\\
        \(\begin{aligned}
            &H \gets (1- \alpha_z)H + \alpha_z \widehat{Z}\\
            &H^{\wght} \gets (1- \alpha_z)H^{\wght} + \alpha_z \widehat{Z}^{\wght}
        \end{aligned}
        \)
        
        \Return{\(\widehat{Z}, \, \widehat{Z}^{\wght}, \, H, \, H^{\wght}\)}
    }
\caption{\algoname:  A compressed approximate Newton-type distributed optimization algorithm}
\label{alg:sec_ord_comp}
\end{algorithm}

\subsection{Convergence analysis}
\label{subsec:convergence}
We establish linear convergence of \(\algoname\) in this section. Define the errors \ \(\opterr(t) \Let \EE \expecof[\Big]{\norm{\avgdmat(t) - \dvar^{\ast}}^2}\) (optimization error),  \( \conerr(t) \Let \EE \expecof[\Big]{\norm{\dmat(t) - \ones \avgdmat(t)}^2}\) (consensus error), \(\gterr (t) \Let \EE \expecof[\Big]{\norm{\gtmat(t) - \ones\avggtmat(t)}^2}\) (gradient tracking error), and the compression errors \(\comerrdv(t) \Let \EE \expecof[\Big]{\norm{\dmat(t) - \auxdv(t)}^2} \) and  \(\comerrgt(t) \Let \EE \expecof[\Big]{\norm{\dmat(t) - \auxgt(t)}^2}\). 
We are ready to present the main theoretical results. We start with following lemma:
\begin{lemma}
    \label{lem: key lemma}
    Consider Algorithm \ref{alg:sec_ord_comp} along with its associated data and notations established in \S\ref{sec:algo}. Suppose that Assumptions \ref{assum:on graph} and \ref{assum:Standard assumptions} hold. Let \(\Wght\) be the consensus weight matrix. Define the vector of errors \(\err(t) \Let \bigl(\opterr(t) \; \conerr(t) \; \gterr(t) \; \comerrdv(t) \; \comerrgt(t) \bigr) \in \Rbb^{\dimdv \times 1}\) for each iteration \(t \in \Nz\) and the vector of the hyper-parameters \(\hypar \Let \bigl(\stsz \; \constsz \; \condv \; \congt \bigr),\) where \(\constsz \in \lorc{0}{1}\) is the consensus step-size, \(\scale>0\) in Definition \ref{def:General Compression Operator} is the scaling parameter associated with the compression scheme, and the parameters satisfy \(\condv, \congt \in \lorc{0}{\frac{1}{\scale}}\). Let \(\wt{\specnorm} \Let (1-\constsz) + \constsz \specnorm\), where \(\specnorm \Let \norm{\Wght   - \frac{1}{\agents}\ones \ones^{\top}}\), \(\Wght \in \Rbb^{\agents \times \agents}\) refers to the weight matrix in Assumption \ref{assum:on graph}, and \(\Cgen>0\) denotes the constant in Definition \ref{def:General Compression Operator}. Define the quantity \(\quanto \Let \norm{\identity{\agents} - \Wght}\). If the condition 
    \(\stsz \leq \min \aset[\big]{\frac{2\lips}{3 \strconv}, \frac{\strconv}{\lips}}\)
    holds, then \(\err(t )\), for each iteration \(t \in \Nz\), is governed by a linear system of inequalities \(\err(t+1) \le \conmat(\hypar) \err(t)\), where the matrix \(\conmat(\hypar)\) is defined in \eqref{eq:final A}, and provided that the following conditions hold:
    \begin{figure*}[!t]
    \begin{align}\label{eq:final A}
        \conmat(\hypar) = \begin{pmatrix}
            1 - \frac{3 \stsz \strconv}{2 \lips} + \frac{\stsz^3 \strconv^3}{2 \lips^3} & \frac{\stsz^2 \lips^2}{\strconv^2 \agents}+\frac{2 \stsz \lips^3}{\strconv^3 \agents} & \frac{\stsz^2 }{\strconv^2 \agents} + \frac{2 \stsz \lips}{\strconv^3 \agents} & 0 & 0 \vspace{1.5mm}\\
            \frac{8 \lips^2 \stsz^2 \agents}{\strconv^2(1-\wt{\specnorm})} & \frac{1+\wt{\specnorm}^2}{2} + \frac{8 \lips^2 \stsz^2 }{\strconv^2 (1- \wt{\specnorm})} & \frac{4 \stsz^2 }{\strconv^2 (1- \wt{\specnorm})} & \frac{2 \constsz^2 \quanto^2 \Cgen^2 }{1-\wt{\specnorm}} & 0 \vspace{1.5mm}\\            
            \frac{24 \lips^4 \stsz^2  \agents}{\strconv^2(1 - \wt{\specnorm})} & \frac{6 \lips^2 \constsz^2  \quanto^2}{1 - \wt{\specnorm}} + \frac{24 \lips^4 \stsz^2}{\strconv^2(1 - \wt{\specnorm})} & \frac{1+ \wt{\specnorm}^2}{2} + \frac{12 \lips^2 \stsz^2}{\strconv^2(1 - \wt{\specnorm})} & \frac{6 \lips^2 \constsz^2  \quanto^2 \Cgen}{1 - \wt{\specnorm}} & \frac{2 \constsz^2 \quanto^2 \Cgen}{1-\wt{\specnorm}} \vspace{1.5mm} \\              
            \frac{4 \lips^2 \stsz^2  \agents  \tdv}{\strconv^2} & \constsz^2 \quant{1} + \frac{4 \lips^2 \stsz^2 \tdv}{\strconv^2} & \frac{2 \stsz^2 \tdv}{\strconv^2} & \quantdv + \constsz^2 \quant{2} & 0 \vspace{1.5mm}\\
            \frac{12 \lips^4 \stsz^2 \agents \tgt}{\strconv^2} & 3 \lips^2 \constsz^2  \quant{3} + \frac{12 \lips^4 \stsz^2 \tgt }{\strconv^2} & \constsz^2 \quant{3} + \frac{6 \lips^2 \stsz^2 \tgt}{\strconv^2} & 3 \lips^2 \constsz^2 \quant{4} & \quantgt + \constsz^2 \quant{3}
        \end{pmatrix}.
    \end{align}
    \end{figure*}
    \begin{align}\label{eq:notations}
        \begin{cases}
            \tdv \Let \frac{3\taudv}{\taudv-1} >1, \, \tgt \Let \frac{3 \taugt}{\taugt-1} >1,\\
            \quantdv = \taudv(1 - \condv \scale \delta) < 1 , \, \quantgt = \taugt(1 - \congt \scale \delta) < 1,\\
            \quant{1} = \tdv \quanto^2, \, \quant{2} = \tdv\Cgen \quanto^2,
            \quant{3} = \tgt  \quanto^2, \, \quant{4} = \tgt \Cgen \quanto^2,
        \end{cases}
    \end{align}
    for \(1 < \taudv < \frac{1}{1 - \condv \scale \delta}\), \(1 < \taugt < \frac{1}{1 - \congt \scale \delta}\), and for \(\delta \in \lorc{0}{1}\).
    \hfill \qedsymbol
\end{lemma}
Readers are referred to Appendix \ref{subappen:main proof} for a detailed proof of Lemma \ref{lem: key lemma}, the outline of which is sketched below.
\subsubsection*{Sketch of the proof of Lemma \ref{lem: key lemma}} 
The proof relies on the fact that if \(\stsz \le \frac{2\lips}{3 \strconv} \), all entries of \(\conmat(\hypar)\) are nonnegative, which is essential for \(\conmat(\hypar)\) to be a positive matrix. Moreover, \(\stsz \le \frac{\strconv}{\lips}\) ensures  global linear convergence of the damped  Newton update step. Hence, \(\stsz \leq \min \aset[\big]{\frac{2\lips}{3 \strconv}, \frac{\strconv}{\lips}}\); for more detail refer to Appendix \ref{subappen:main proof}. The idea of the proof, adopted from \cite{ref:GQ-NL-17}, involves upper-bounding each of the errors in \(\err(\cdot)\) in terms of their values computed at the previous instant. To bound the optimality error, we utilize Assumption \ref{assum:Standard assumptions}, and establish upper and lower bounds of the term \(\frac{1}{\agents} \sum_{i=1}^{\agents} \hess \objfunc_i(\dvar_i(t)){\inverse} \Bigl( \grd \objfunc \bigl(\avgdmat(t)^{\top}\bigr)  - \grd \objfunc(\dvar^{\ast})\Big)\), which is subsequently used to compute the optimality error \((\opterr(t))_{t \in \Nz}\). Computation of the consensus error \((\conerr(t))_{t \in \Nz}\) depends on the assertion that \(\avggtmat(t) = \frac{1}{\agents} \ones^{\top} \grd \Objfunc(\dmat(t)) = \grd \av{\Objfunc}\big( \dmat(t)\big)\) for each \(t\). Calculation of the gradient tracking error relies on the estimates of \(\norm{\gtmat(t)}\) and \(\norm{\grd \Objfunc\big(\dmat(t+1)\big) - \Objfunc\big(\dmat(t)\big)}\). Finally, the compression errors \((\comerrdv(t), \comerrgt(t))\) at each \(t
\) are estimated by extensively using \eqref{eq:general compressor operator}.

The next result establishes a linear convergence rate.
\begin{theorem}
    \label{thm:key thm}
    Let the hypothesis of Lemma \ref{lem: key lemma} hold. If, \((\stsz, \constsz)\) is chosen such that \(\spec\bigl(\conmat(\hypar)\bigr) < 1\), then the sequence  \((\err(t))_{t \in \Nz}\) converges to zero at a linear rate \(\Bigoh\Bigl(\spec\bigl(\conmat(\hypar) \bigr)^t\Bigr)\). \hfill \qedsymbol
\end{theorem}
\begin{proof}
    If \(\stsz \le \frac{2\lips}{3 \strconv} \), all entries of \(\conmat(\hypar)\) are nonnegative, implying that \(\conmat(\hypar)^t\) for each \(t\) is nonnegative. From the hypothesis, \((\stsz,\constsz)\) is chosen such that \(\spec(\conmat(\hypar)) < 1\), then from \cite[Theorem \(5.6.12\)]{ref:RAJ-CRJ-13} we deduce that \(\conmat(\hypar)^t \xrightarrow[t \ra +\infty]{} 0\) for every \((\condv, \congt)\).  Consequently,  for every initial condition \(\err(0)\) and every \((\condv, \congt)\), \(0 \le \limsup_{t \uparrow + \infty} \err(t) \le \liminf_{t \uparrow + \infty} \conmat(\hypar)^t \err(0) \xrightarrow[t \ra +\infty]{}0\) at a linear rate \(\Bigoh\big(\spec(\conmat(\hypar))^t\big)\). The proof is now complete.
\end{proof}
\begin{remark}[On the linear rate of convergence]
    \label{rem:linear rate}
    Note that the linear convergence rate established in Theorem \ref{thm:key thm} for \(\algoname\)  depends heavily on \(\lips\) and \(\strconv\), and also on the connectivity parameter \(\specnorm\) of the underlying graph.
\end{remark}
The next result establish a set of sufficient conditions on \((\stsz,\constsz)\) to ensure that \(\spec\bigl(\conmat(\hypar)\bigr) < 1\). Let us denote by 
\[
    \cn \Let \frac{\lips}{\strconv} > 1,
\]
the conditional number of the Hessian of the objective function \(\objfunc\) corresponding to the optimization problem \eqref{eq:key prob}. 
\begin{theorem}
    \label{th:Main result_str_conv_case}
    Suppose that the hypothesis of Lemma \ref{lem: key lemma} holds. Let \((\condv,\congt )\in \lorc{0}{\frac{1}{\scale}} \times \lorc{0}{\frac{1}{\scale}}\). Define the quantities
    \begin{align*}
        \ol{\specnorm} \Let 1 - \wt{\specnorm},
    \end{align*}
    and recall from \eqref{eq:final A} the expression of \(\conmat(\hypar)\).
    If the step-size and the consensus step-size satisfy
    \begin{align}
        \label{eq:stsz_constsz}
        &\stsz \le \min \left\{\frac{\eps_1}{\frac{3 \cn^3 \eps_2}{\agents} + \frac{3 \cn\eps_3}{\strconv^2 \agents}}, \cn \sqrt{\frac{2}{3}}, \frac{\constsz (1 - \specnorm)\cn}{6}, \frac{\constsz}{\cn},\cdots \right. \\ 
        &\left.\cdots \frac{1}{4 \cn } \sqrt{\frac{\constsz(1 - \specnorm)\ol{\specnorm} \eps_2}{\wh{\epsln}}}, \frac{1}{12 \cn }\sqrt{\frac{\constsz(1 - \specnorm)\ol{\specnorm} \eps_3}{\wh{\epsln}}}\right\} \nn\\
        & \constsz \le \min \aset[\Bigg]{1, \sqrt{\frac{(1 - \quantdv)\eps_4}{2 \tdv \wh{\epsln} + \ol{\epsln} + \frac{\eps_4}{2 \cn^2}}}, \sqrt{\frac{(1 - \quantgt)\eps_5}{6 \tgt \wh{\epsln} +  \Breve{\epsln} + \frac{\eps_5}{2 \cn^2}}} }
    \end{align}
    where \(\wh{\epsln} \Let 2 \agents \eps_1 + 2 \eps_2 + \eps_3\), \(\ol{\epsln} \Let \quant{1} \eps_2 + \quant{2}\eps_4\), \(\Breve{\epsln} \Let 3 \quant{3} \eps_2 + \quant{3}\eps_3 + 3 \quant{4}\eps_4 + \quant{3}\eps_5\), and
    \(\eps \Let \big(\eps_1, \eps_2, \lips^2 \eps_3, \eps_4, \lips^2 \eps_5 \big)\) is a positive vector with \(\eps_i > 0\) for each \(i =1,2,\ldots,5\), such that 
    \begin{align}\label{eq:system_ineq}
    \begin{aligned}
        &\frac{\eps_1}{\eps_2} \ge \frac{6 \cn^4}{\agents} + \frac{6 \cn^2}{\strconv^2 \agents} \frac{\eps_3}{\eps_2}, \, \frac{\eps_4}{\eps_2} \le \frac{\ol{\specnorm}(1 - \specnorm)}{8 \constsz^2 \quanto^2 \Cgen^2}, \, \text{and } \\
        &3 \eps_2 + 3 \Cgen \eps_4 + \Cgen \eps_5 \le \frac{\ol{\specnorm}(1 - \specnorm)\eps_3}{24 \constsz \quanto^2}. 
        \end{aligned}
    \end{align}
    Then \(\spec\big(\conmat(\hypar) \big)  < 1\).
    \end{theorem}
The proof is deferred to Appendix \ref{subappen:main proof strong convex}.

\noindent We acknowledge that the conditions in Theorem \ref{th:Main result_str_conv_case} may not be tight. Nevertheless, they serve two purposes: first, they guarantee the feasibility of the linear system of inequalities established in Lemma \ref{lem: key lemma}; and second, they highlight the bottlenecks introduced by various factors --- such as the condition number of the Hessian, the spectral radius of the underlying network, and the chosen compression scheme --- on the step-size \(\stsz\) and the convergence rate.
\section{Experiments}
\label{sec:experiments}
Empirical assessment of the performance of \(\algoname\) against gradient-based algorithms \(\cgt\), \(\lead\), \(\cold\), and approximate Newton-based algorithms \cite{ref:HL-JZ-AMCS-QL-23} (which we refer to as \(\comphess\) in our subsequent discussions) is carried out keeping the centralized Newton-Raphson algorithm as the baseline. The experiments were conducted on two classes of problems --- distributed regularized ridge regression and distributed binary classification with regularized logistic regression --- under different compression schemes. 
\subsection{Compression schemes used}\label{subsec:comp employed}
\begin{itemize}[leftmargin=*, label = \(\circ\)]
    \item \textit{Unbiased compression:} \(q\)-norm \(b-\)bits quantization \cite[Assumption \(2\)]{ref:XL-YL-RW-JT-MY-21} (qNbB-Q) given by \(\ucompop(\dvar) = \big(\norm{\dvar}_{\infty} 2^{-(b-1)} \text{sign}(\dvar) \big)\hada \left \lfloor \frac{2^{(b-1)} \abs{\dvar}}{\norm{\dvar}_{\infty} }+ u\right \rfloor\),
    where \(b =2\) and \(q = + \infty\) are considered, \(u\) is a uniformly distributed in \([0,1]^{\dimdv}\), and `\(\circ\)' refers to the standard Hadamard product. It is reported in \cite{ref:AK-SS-MJ-19} that the number of bits per iteration required during transmission is equal to \((1 + b)\dimdv \agents t\), where \(t \in \Nz\) is the iteration.;
    \item \textit{Biased compression \cite[\S~\(2.2\)]{ref:AB-SH-PR-MS-23}:} We considered the Random-\(k\) sparsification (Rk-S) method, given by the expressions \(\bcompop(\dvar) = \dvar \hada e\), where a set of \(k\)-elements of \(\dvar\) are transmitted, and the elements of \(e\) satisfies: \(e_i = 1\) w.p. \(\frac{k}{\dimdv}\) and \(e_i = 0\) w.p. \(1 - \frac{k}{\dimdv}\),
    and the Top-\(k\) sparsification (Tk-S) method, given by \(\bcompop(\dvar) = \dvar \hada e \), where the elements of \(\dvar\) is arranged in decreasing order and a subset of \(k\) largest absolute values are considered. Here the vector \(e\) satisfies: \(e_{i_l} = 1\) for \(l \le k\) and \(e_{i_l} = 0\) for \(l \ge k\). In the experiments we pick \(k\) to be \(5\) and \(3\) for the random-\(k\) and the top-\(k\) compression schemes, respectively. For Random-\(k\), the number of bits per iteration required during transmission is
    equal to \((32 + \left \lceil \log_2 \dimdv \right \rceil) k\agents t\) (assuming that \(32\) bits are utilized to encode each entries), and for Top-\(k\), we require twice as many bits, i.e., \((64 + \left \lceil \log_2 \dimdv \right \rceil) k\agents t\). Top-k when implemented for \(\comphess\) with singular value decomposition requires \(64k(1+\dimdv)\agents t\) number of bits per iterations for transmission.
    \item \textit{General compression \cite[\S~\(II\)]{ref:YL-ZL-KH-SP-22}: } The \(q\)-norm-signed compression (qNS-C) scheme defined by \(\gcompop(\dvar) = \norm{\dvar}_q \text{sign}(\dvar)\) is considered. Throughout all the experiments, we fix \(q = +\infty\). For this compression, the number of bits per iteration required during transmission is
    equal to \((\dimdv +32)\agents t\) with the number of bits being \(32\); see \cite[Appendix C]{ref:AK-TL-SUS-MJ-19}.
\end{itemize}

\subsection{Ridge regression with synthetic data}\label{subsec:rr}

\subsubsection*{Problem setup}
The problem under consideration is an unconstrained distributed optimization problem given by 
\begin{align}
    \label{eq:ridge regression}
    \min_{\dvar \in \Rbb^{\dimdv}} \objfunc(\dvar) = \frac{1}{\agents} \sum_{i=1}^{\agents} \norm{A_i \dvar - B_i}^2 + \regu \norm{\dvar}^2
\end{align}
where \(\regu>0\) is the regularizer, \(A_i \in\Rbb^{m_i \times \dimdv}\), \(b_i \in \Rbb^{m_i \times 1}\) for each \(i \in \aset[]{1,2,\ldots, \agents}\), respectively. Here \(m_i\) is the number of samples assigned to the \(i\)-th agent. The optimization problem \eqref{eq:ridge regression} admits a unique solution given by \(\dvar^{\ast} = \big(\sum_{i=1}^{\agents}A_i A_i^{\top} + \agents \regu \identity{\dimdv}\big){\inverse} \sum_{i =1}^{\agents}A_ib_i\). 

For the experiments, we considered an undirected graph having \(\specnorm =   0.8727\), with \(\agents =20\) agents, connected in a ring network. We assumed that each agent can exchange data with their \(1\)-hop neighbors. The consensus weight matrix was constructed using the Metropolis-Hastings algorithm \cite{ref:LX-SB-SJK-07}.

\subsubsection*{Data generation and distribution}
For the experiment we generated \(N\) many samples \((a_j,b_j) \in \Rbb^{1 \times \dimdv} \times \Rbb\), where \(j = 1,2, \ldots, N\). These were generated by following the steps enumerated in \cite[\S5]{ref:YL-ZL-KH-SP-22}.
The generated data was further randomly reshuffled to reduce bias and distributed homogeneously among the agents such that each agent received \(m_i = \left \lfloor \frac{N}{\agents}\right \rfloor\) number of samples. Throughout, we picked \(p = 20\), \(\agents = 10\), \(N = 500\).\footnote{For reproducibility, we fix the seed used to generate the random variables to \(42\) across all experiments and reported results.}

\subsubsection*{Initialization}
The variables \(\dmat(0), \auxdv(0), \auxgt(0), \gtmat(0)\) were initialized uniformly. We imposed \(\grd \Objfunc(\dmat(0)) = \gtmat(0)\),  \(\wauxdv(0) = \Wght \auxdv(0)\), and  \(\wauxgt(0) = \Wght \auxgt(0)\). 

\subsubsection*{Hyper-parameters}
We reveal that the hyper-parameters were selected to obtain the best performance for the algorithms. We fixed \(\regu = 0.5\),  the number of iterations to \(5000\), the consensus step-size \(\constsz = 0.6\), and \((\condv, \congt)\) to be \((1,1)\) across all experiments. Table \ref{tab:prog-comp1} below lists the hyper-parameters chosen for the experiments.
\begin{table}[htpb]
\centering
\begin{adjustbox}{max width=\linewidth}
\begin{tblr}{lccccc}
\hline[2pt]
\SetRow{azure9}
Algorithm & qNbB-Q & Rk-S & Tk-S & qNS-C \\
\hline[1pt]
C-GT \cite{ref:YL-ZL-KH-SP-22} 
    & 0.013 & 0.013 & 0.0015 & 0.0112 \\
\algoname 
    & 0.0095 & 0.0012 & 0.006 & 0.021 \\
\hline[2pt]
\end{tblr}
\end{adjustbox}
\caption{Step-size \(\stsz\) for different compression schemes.}
\label{tab:prog-comp1}
\end{table}
\subsubsection*{Results and discussions}
The performance of \(\algoname\) was compared with \cite{ref:YL-ZL-KH-SP-22} on the synthetic data set. As mentioned above, we choose the parameter such that the best performance is obtained. Figure \ref{fig:fourplots} summarizes our findings. We observed that \(\algoname\) under qNS-C is the most communication-efficient algorithm, followed by qNbB-Q. Moreover, we draw attention to the fact that \(\algoname\) consistently outperformed \(\cgt\) across all the considered compression schemes. Comparing \(\algoname\) with the uncompressed algorithms Network-GIANT and \cite{ref:GQ-NL-17} (see Figures \ref{subfig:qnb_con}-\ref{subfig:qnb_gt}), we observed that \(\algoname\) performs on par with them when the comparison is done with respect to \emph{iterations}, consistent with the theoretical claims. This is expected because the compression error converges to zero asymptotically. 

\begin{figure*}[t]
    \centering
    \begin{subfigure}{0.22\textwidth}
        \includegraphics[width=\linewidth, height = 2.6cm]{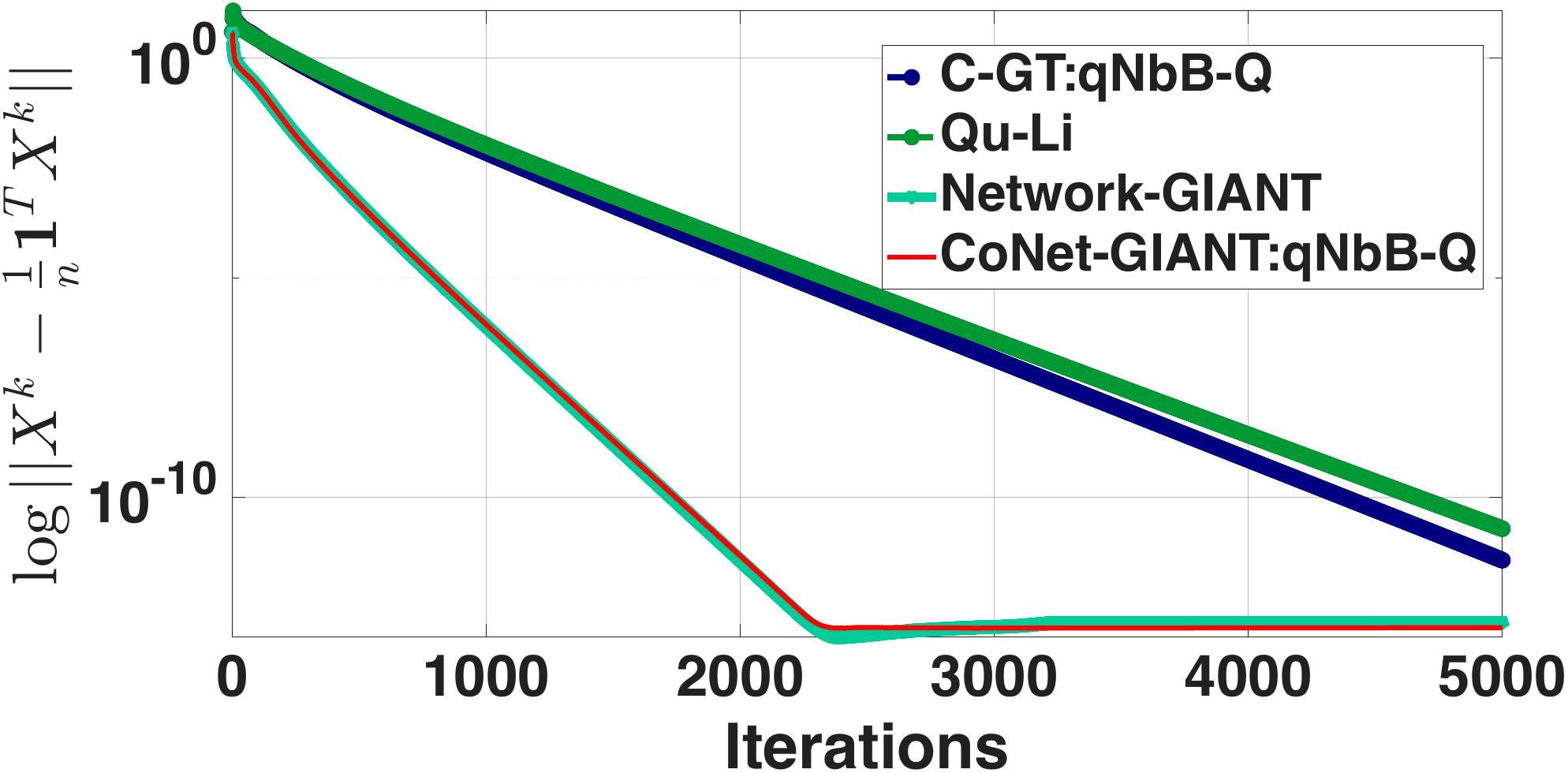}
        \caption{}
        \label{subfig:qnb_con}
    \end{subfigure}\hfill
    \begin{subfigure}{0.22\textwidth}
        \includegraphics[width=\linewidth,height = 2.6cm]{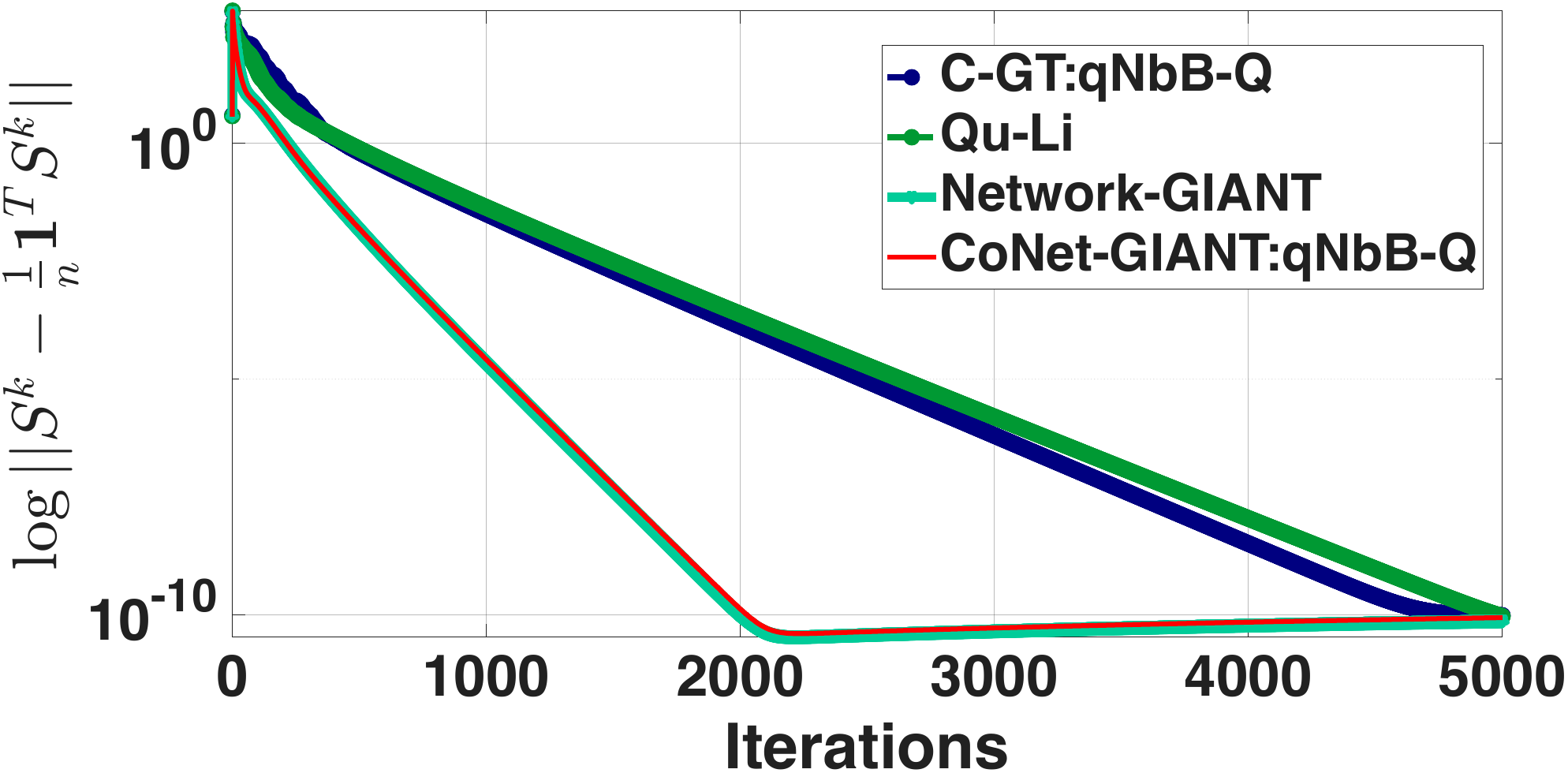}
        \caption{}
        \label{subfig:qnb_opt}
    \end{subfigure}\hfill
    \begin{subfigure}{0.22\textwidth}
        \includegraphics[width=\linewidth,height = 2.6cm]{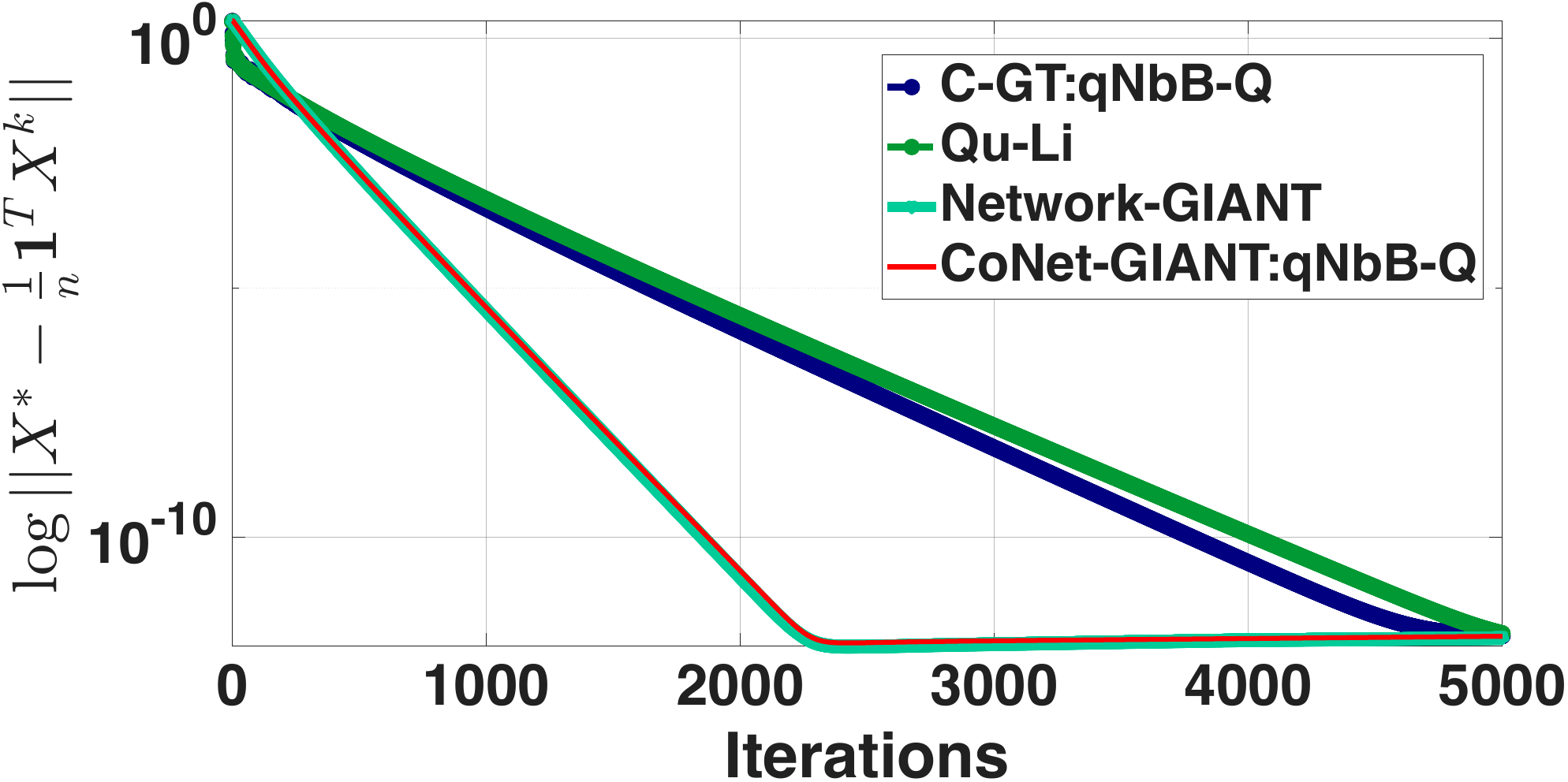}
        \caption{}
        \label{subfig:qnb_gt}
    \end{subfigure}\hfill
    \begin{subfigure}{0.28\textwidth}
        \includegraphics[width=\linewidth,height = 2.8cm]{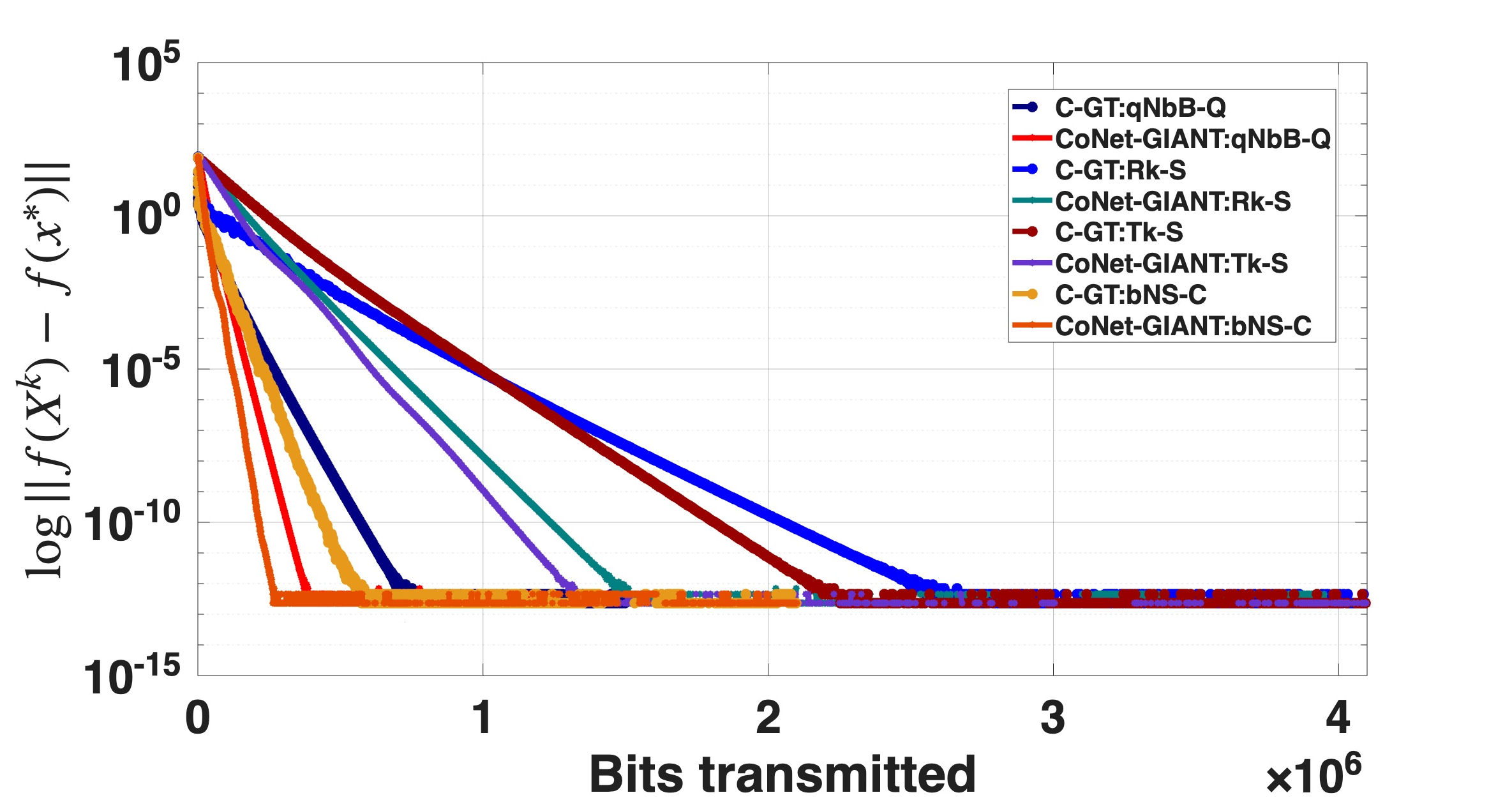}
        \caption{}
        \label{subfig:qnb_residual}
    \end{subfigure}
    
    \caption{Figures \ref{subfig:qnb_con}-\ref{subfig:qnb_gt} compare the errors admitted by \(\algoname\), \(\cgt\),  Network-GIANT, and Qu-Li \cite{ref:GQ-NL-17} for qNbB-Q compression scheme. Figure \ref{subfig:qnb_residual} plots the optimality gaps in terms of the functional values of the global objective function. }
    \label{fig:fourplots}
\end{figure*}

\subsection{Binary logistic classification}
\subsubsection*{Problem setup}
Given the local data \(\mathcal{D}_i \Let \aset[\Big]{\big(u^i_j,v^i_j\big) \in \Rbb^{1 \times \dimdv} \times \aset[]{-1,1}}_{j=1}^{m_i}\) available to the \(i\)-th agent, the following regularized optimization problem
\begin{align*}
    \label{eq:log_bin}
    \min_{\dvar \in \Rbb^{\dimdv}} \objfunc_i(\dvar) &= \frac{1}{m_i}\sum_{i=1}^{m_i} \log\Big( 1 + \exp\big(-v^i_j(\dvar^{\top}u^i_j) \big)\Big) + \frac{\regu}{2} \norm{\dvar}^2
\end{align*}
 is solved. We considered two graph topologies for the experiments: first, an undirected graph having \(\specnorm =   0.8727\), with \(\agents = 10\) agents, connected in a ring network with self-loops. As before, the Metropolis-Hastings algorithm was used to construct the consensus weight matrix, and assume that each agent can exchange data with their \(1\)-hop neighbors. Second, we considered an undirected expander graph with the number of nodes \(n = 14\), with each node having a degree \(6\). Such a graph has a spectral norm \(\specnorm = 0.7912\). We constructed the consensus weight matrix using the Metropolis-Hastings algorithm, as before.

\subsubsection*{Data generation and distribution}
We performed distributed classification on the CovType dataset \cite{ref:DD-CG-19}. We applied standard principal component analysis (PCA) to reduce the dataset, resulting in \(566602\) sample points and \(p = 10\) features. We randomly reshuffled the samples to minimize bias, divided the entire dataset into two parts --- training data \((\mathrm{U}_{\text{train}}, \mathrm{V}_{\text{train}})\) and testing data \((\mathrm{U}_{\text{test}}, \mathrm{V}_{\text{test}})\) --- and distributed homogeneously among the agents such that each agent received \(m_i = \left \lfloor \frac{N}{\agents}\right \rfloor\) number of samples. Throughout, we picked the number of nodes to be \(\agents = 10\).
\begin{table}[!htpb]
\centering
\begin{adjustbox}{max width=\linewidth}
\begin{tblr}{l c c c}
\hline[2pt]
	\SetRow{azure9}
 & Training samples  & Testing samples & Accuracy \\
\hline[1pt]
Ring network & 400000 & 166602 & 0.5984 \\
 \hline[0.5pt]
Expander network (\(\agents = 14\)) & 400008 & 166594 & 0.6014 \\
\hline[2pt]
\end{tblr}
\end{adjustbox}
\caption{Number of training and testing samples.}
\label{tab:prog-comp4}
\end{table}

\subsubsection*{Initialization} All variables were initialized identically, as it was done for the ridge regression problem in \S~\ref{subsec:rr}.
\subsubsection*{Hyper-parameters}
\emph{The hyper-parameters were selected to obtain the best performance for the algorithms.} We fixed the regularizer \(\regu = 0.1\) and the number of iterations to \(1000\). 
\begin{table}[htpb]
\centering
\begin{adjustbox}{max width=\linewidth}
\begin{tblr}{
  colspec = {l |cccc|cccc|cccc}, 
  row{1,2} = {azure9},
  hline{1,3,Z} = {2pt},
  hline{2} = {1pt},
  column{2-5,6-9,10-13} = {c}, 
}
Algorithm &
\SetCell[c=4]{c} qNbB-Q & & & &
\SetCell[c=4]{c} Tk-S & & & &
\SetCell[c=4]{c} qNS-C & & & &\\
 & Ring \(\constsz\) & Ring \(\stsz\) & Exp. \(\constsz\) & Exp. \(\stsz\) &
   Ring \(\constsz\) & Ring \(\stsz\) & Exp. \(\constsz\) & Exp. \(\stsz\) &
   Ring \(\constsz\) & Ring \(\stsz\) & Exp. \(\constsz\) & Exp. \(\stsz\) \\

\(\cgt\) \cite{ref:YL-ZL-KH-SP-22} &
0.35 & 0.10 & 0.20 & 0.10 &
0.65 & 0.05 & 0.21 & 0.10 &
0.35 & 0.15 & 0.30 & 0.10 \\

\(\lead\) \cite{ref:XL-YL-RW-JT-MY-20} &
0.60 & 0.09 & 0.60 & 0.10 &
0.30 & 0.10 & 0.60 & 0.09 &
-- & -- & -- & -- \\

\(\cold\) \cite{ref:JZ-KY-LX-23} &
0.40 & 0.09 & 0.40 & 0.10 &
0.40 & 0.06 & 0.40 & 0.10 &
0.30 & 0.05 & 0.30 & 0.05 \\

\(\comphess\) \cite{ref:HL-JZ-AMCS-QL-23} &
-- & -- & -- & -- &
0.5 & 0.085 & 0.5 & 0.09 &
-- & -- & -- & -- &\\

\(\algoname\) &
0.35 & 0.093 & 0.20 & 0.09 &
0.40 & 0.098 & 0.21 & 0.08 &
0.35 & 0.095 & 0.30 & 0.095 \\

\hline[2pt]
\end{tblr}
\end{adjustbox}
\caption{\(\stsz\) and \(\constsz\) for different compression schemes.}
\label{tab:prog-comp2}
\end{table}
Moreover, \((\condv, \congt)\) to be \((0.5,0.5)\) across all experiments, both for the ring network and the expander graph. 
Table \ref{tab:prog-comp2} lists the hyper-parameters chosen for the experiments, for the ring network and the expander network.

\subsubsection*{Results and discussions}
In the following experiments, the true optimal solution \(\dvar^{\ast}\) and the optimal value \(\objfunc(\dvar^{\ast})\) correspond to the centralized Newton-Raphson algorithm with backtracking line search. These values were taken as the baseline, and different errors were computed with respect to these baselines. \(\algoname\) was compared with the communication-efficient first-order algorithms --- \(\cgt\), \(\lead\), and \(\cold\) ---  as well as with the \emph{compressed-Hessian} second-order distributed algorithm \cite{ref:HL-JZ-AMCS-QL-23} under the Tk-S compression scheme, which we refer to as \(\comphess\) for convenience. We kept the consensus step-size for \(\comphess\) to be \(1\). We did not compare GIANT because the results would depend on the underlying graph. Consistent with our observations reported in Figure \ref{fig:fourplots}, we observed from Figures \ref{fig:covtype_qnb} and \ref{fig:covtype_qnb_ex} that \(\algoname\) achieves faster convergence at a linear rate and is the most communication-efficient algorithm among \(\cgt\), \(\lead\), \(\cold\), and \(\comphess\). In fact, we found that \(\algoname\):qNbB-Q and \(\algoname\):qNS-C communicated the fewest bits. A similar trend is observed for \(\cgt\) and \(\cold\). Moreover, \(\algoname\) outperformed \(\comphess\) by two orders of magnitude; see Figures \ref{subfig:fc_tk_cov} and \ref{subfig:fc_tk_cov_ex}. In fact, the gradient-based algorithms \(\cgt\), \(\lead\), and \(\cold\) performed better than \(\comphess\) in terms of communication-efficiency. This is because \(\comphess\) compresses the Hessian and therefore requires a larger number of bits for transmission than \(\algoname\), while the computational time \(\comphess\) is identical to \(\algoname\) due to Hessian inversion.

 We further observed that \(\lead\) diverged under the qNS-C compression scheme, which aligns with the empirical findings reported in \cite{ref:YL-ZL-KH-SP-22,ref:JZ-KY-LX-23}, and so we decided to omit the plots for LEAD:qNS-C from Figures \ref{subfig:fc_ns_cov} and \ref{subfig:fc_ns_cov_ex}. We have also removed \(\lead\):Tk-S, \(\cold\):Tk-S, and \(\cold\):qNS-C from Figures \ref{fig:covtype_qnb} and \ref{fig:covtype_qnb_ex}, primarily to remove clutter and because they are not communication-efficient algorithms.

\begin{figure}[!htbp]
    \centering
    \begin{subfigure}{0.48\columnwidth}
        \includegraphics[width=\linewidth,height=2.8cm]{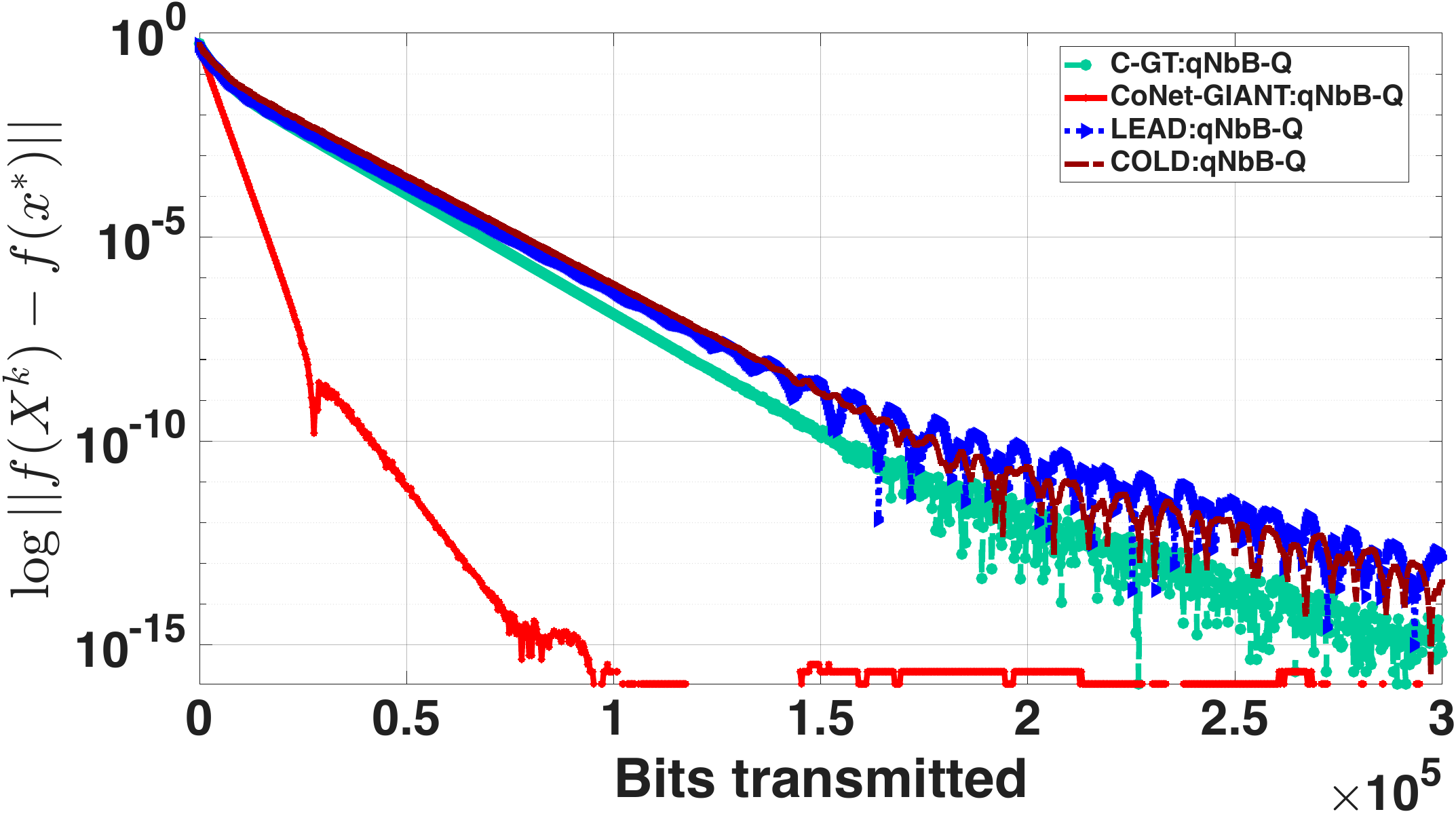}
        \caption{}
        \label{subfig:fc_qnbb_cov}
    \end{subfigure}\hfill
    \begin{subfigure}{0.48\columnwidth}
        \includegraphics[width=\linewidth,height=2.8cm]{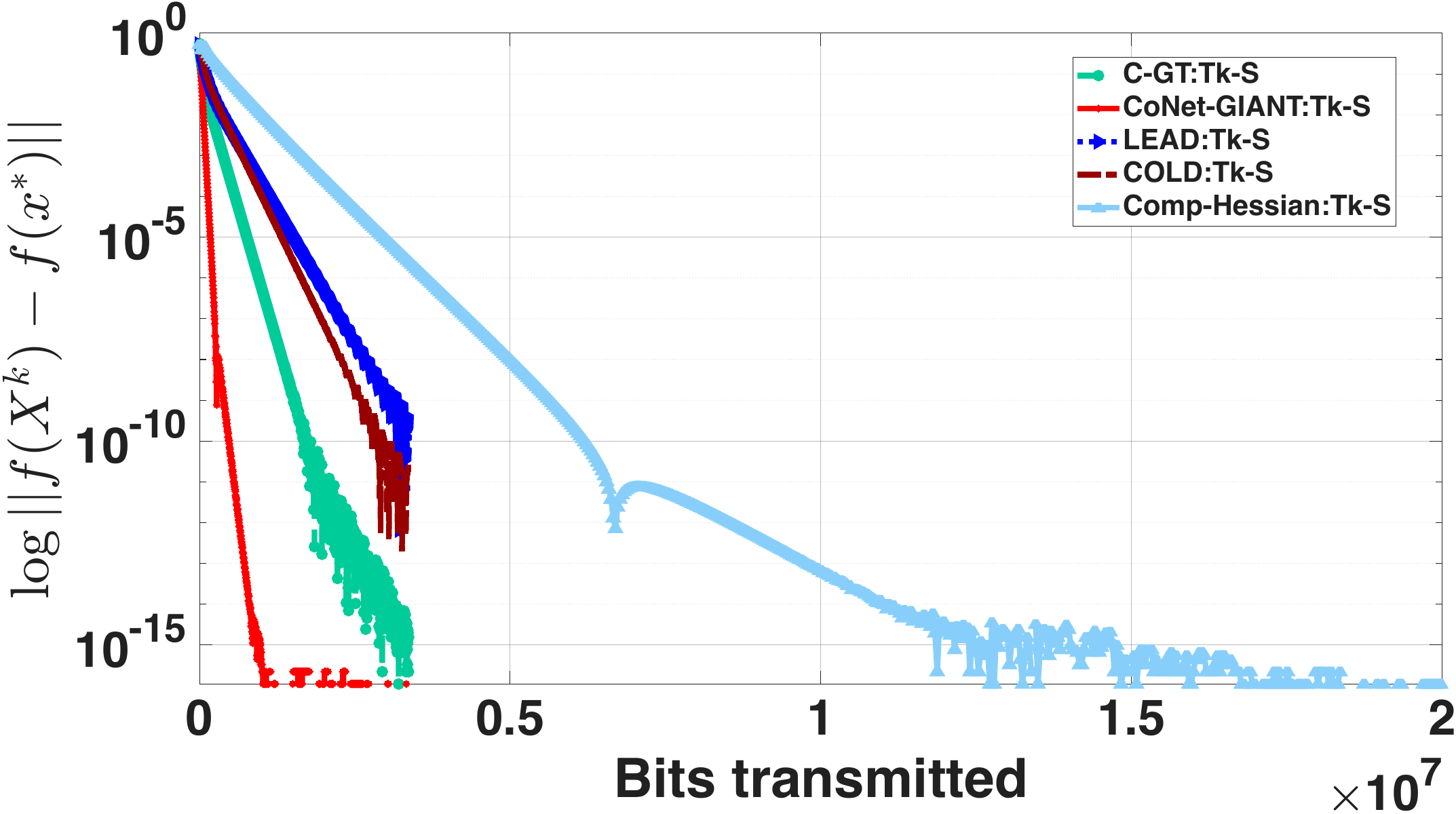}
        \caption{}
        \label{subfig:fc_tk_cov}
    \end{subfigure}

    \begin{subfigure}{0.48\columnwidth}
        \includegraphics[width=\linewidth,height=2.8cm]{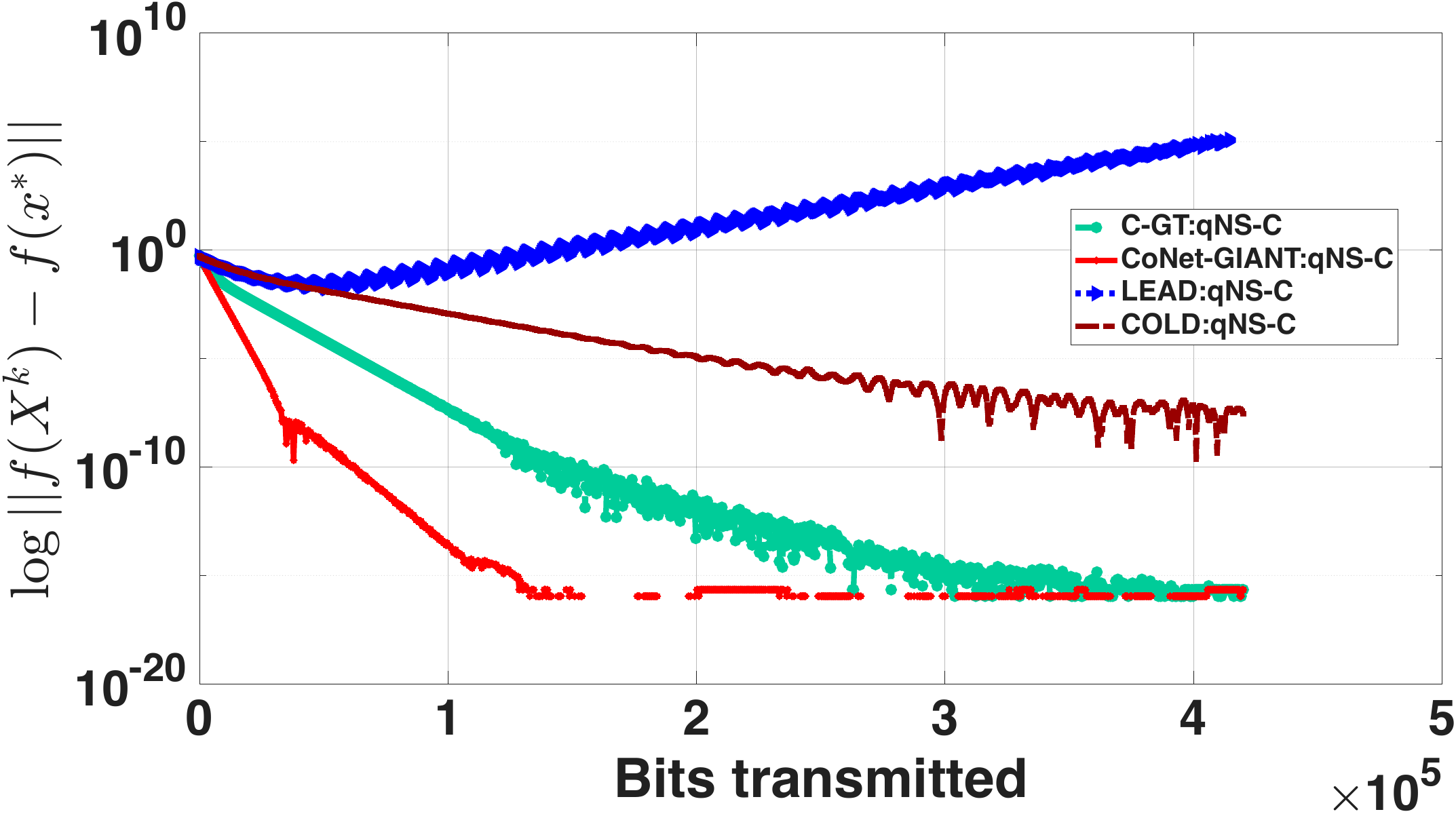}
        \caption{}
        \label{subfig:fc_ns_cov}
    \end{subfigure}\hfill
    \begin{subfigure}{0.48\columnwidth}
        \includegraphics[width=\linewidth,height=2.8cm]{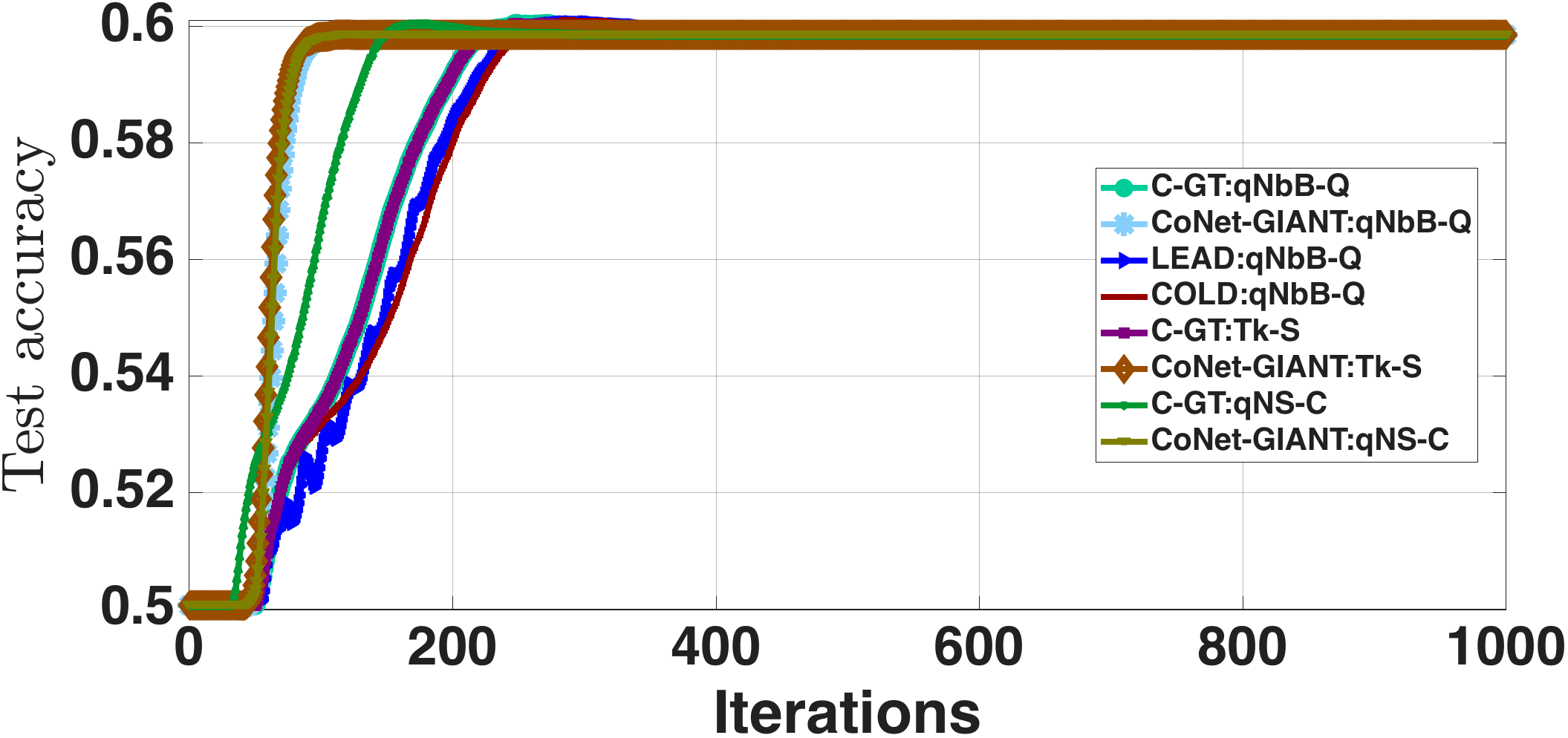}
        \caption{}
        \label{subfig:fc_acc_cov}
    \end{subfigure}

    \caption{Figure \ref{subfig:fc_qnbb_cov}–\ref{subfig:fc_ns_cov} compare \(\algoname\) with \(\cgt\), \(\lead\), \(\cold\), and \(\comphess\) under different schemes for the ring network. Figure \ref{subfig:fc_acc_cov} summarizes the accuracy attained by these algorithms.}
    \label{fig:covtype_qnb}
\end{figure}
\begin{figure}[!htbp]
    \centering
    \begin{subfigure}{0.48\columnwidth}
        \includegraphics[width=\linewidth,height=2.8cm]{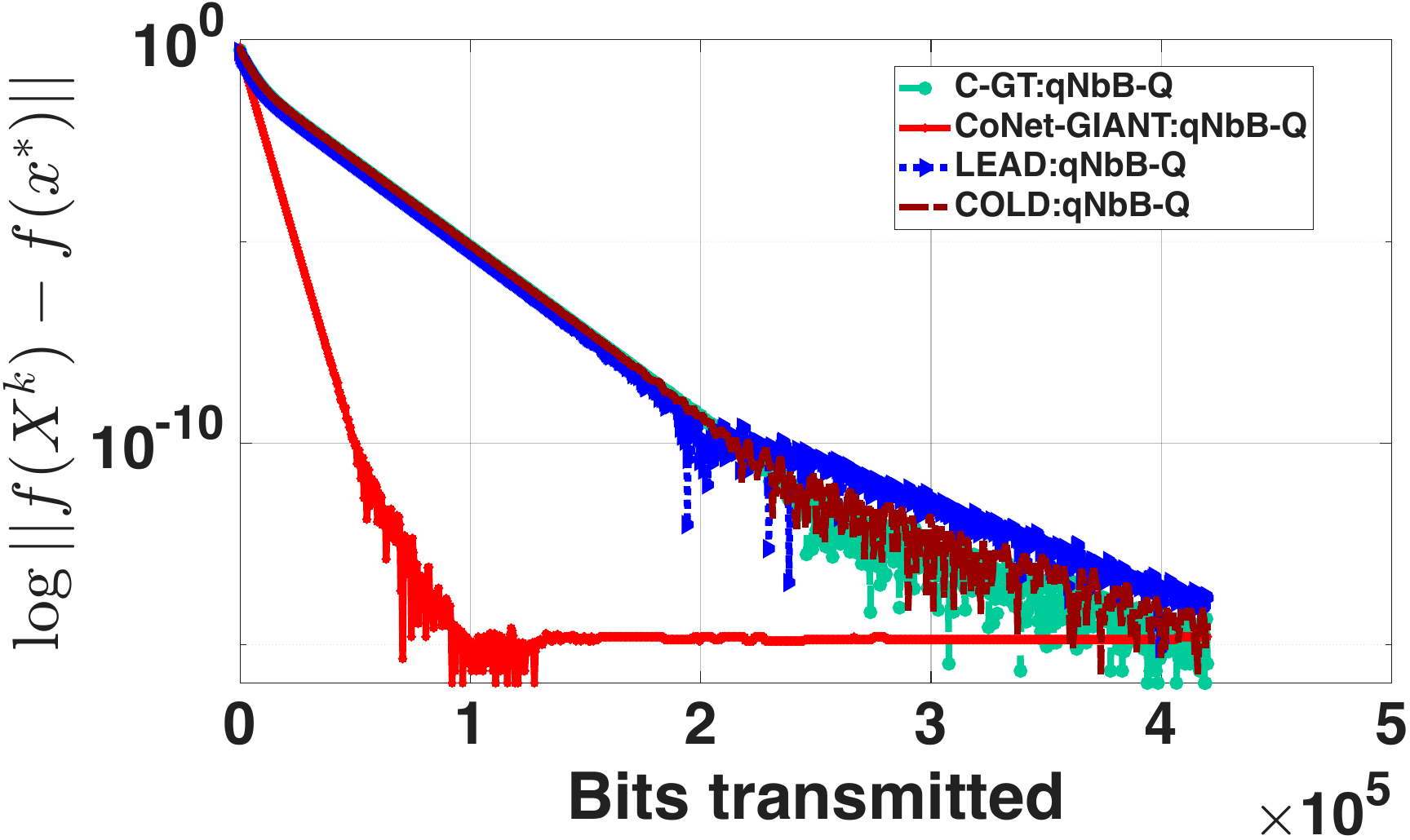}
        \caption{}
        \label{subfig:fc_qnbb_cov_ex}
    \end{subfigure}\hfill
    \begin{subfigure}{0.48\columnwidth}
        \includegraphics[width=\linewidth,height=2.8cm]{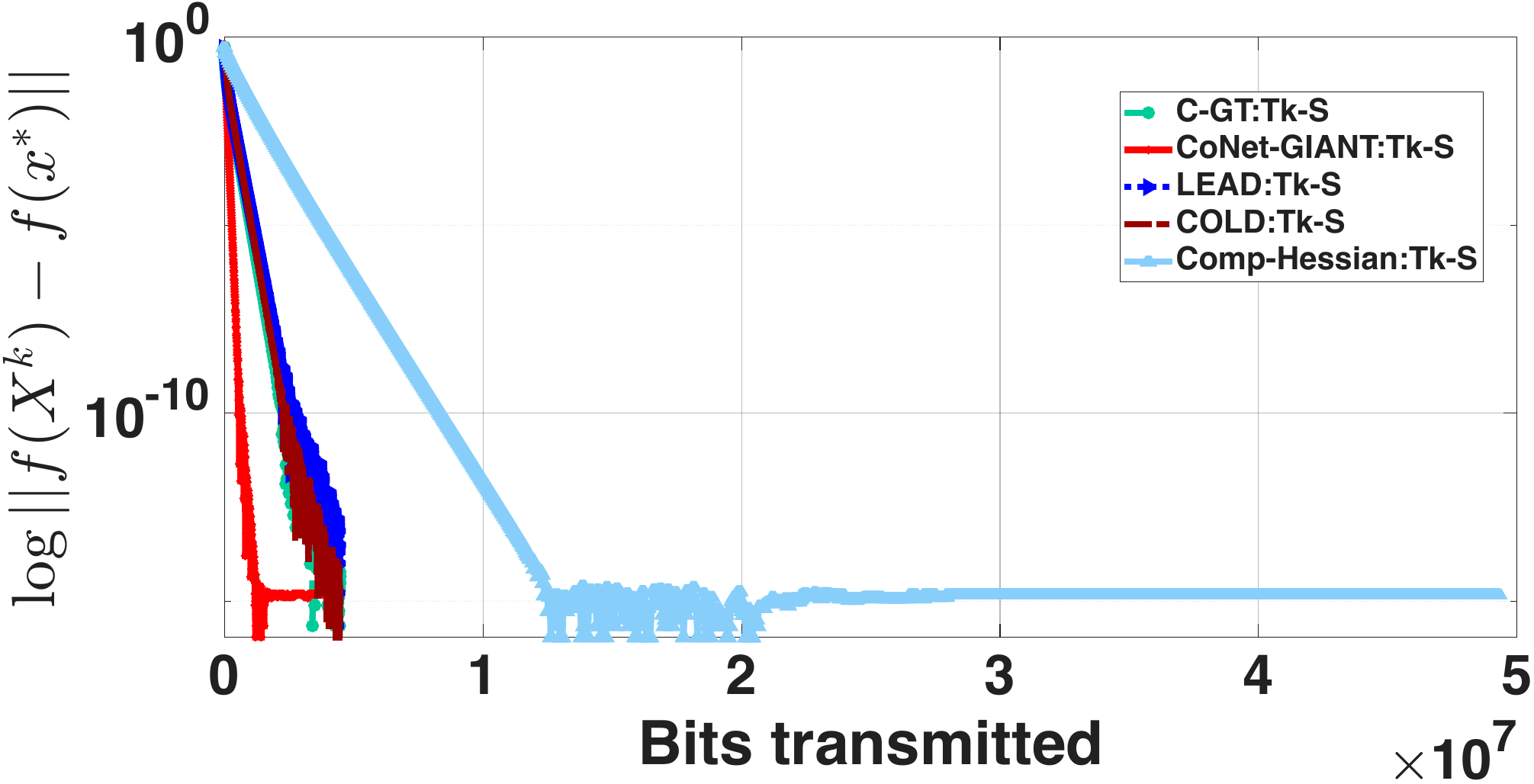}
        \caption{}
        \label{subfig:fc_tk_cov_ex}
    \end{subfigure}

    \begin{subfigure}{0.48\columnwidth}
        \includegraphics[width=\linewidth,height=2.8cm]{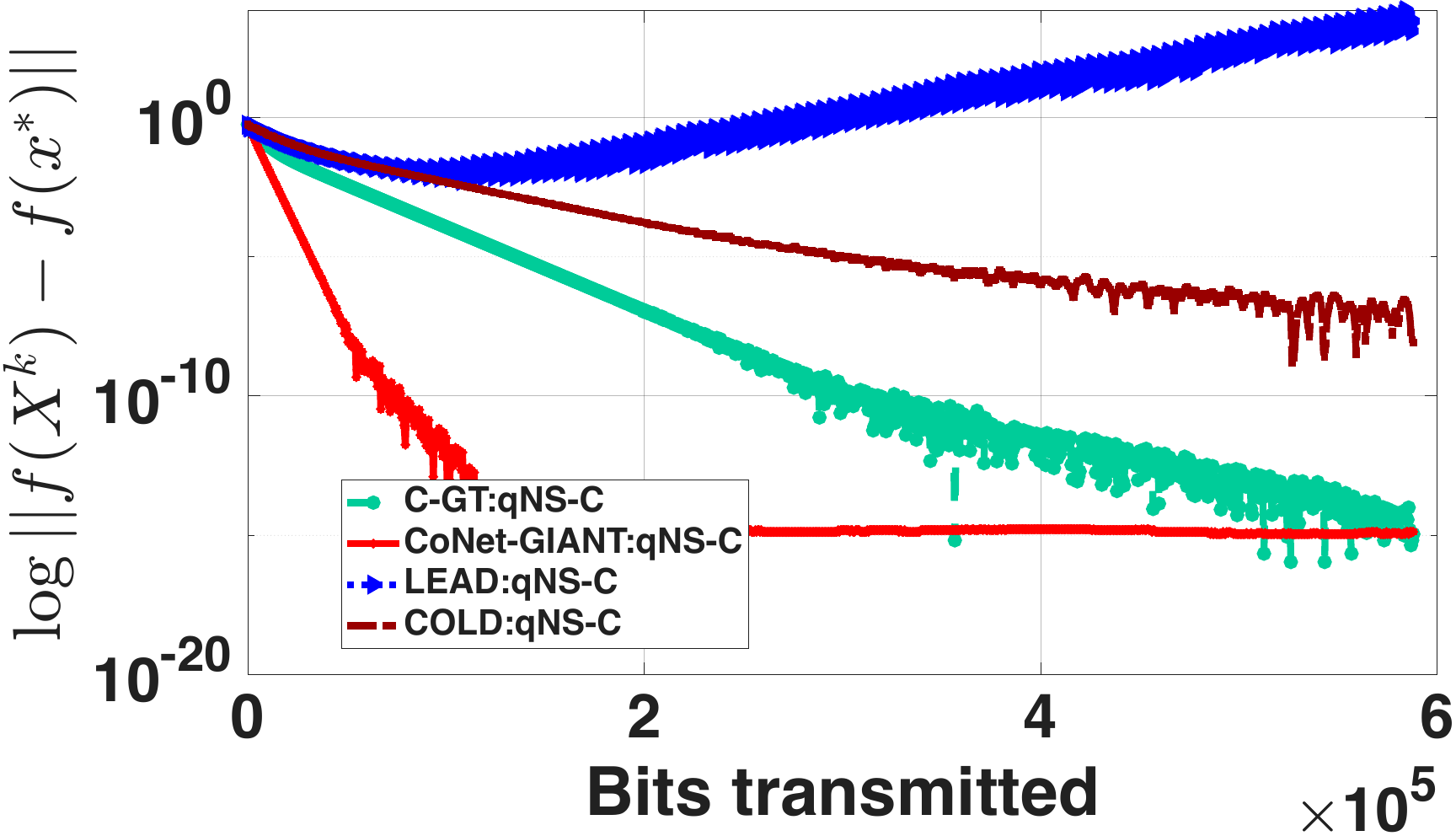}
        \caption{}
        \label{subfig:fc_ns_cov_ex}
    \end{subfigure}\hfill
    \begin{subfigure}{0.48\columnwidth}
        \includegraphics[width=\linewidth,height=2.8cm]{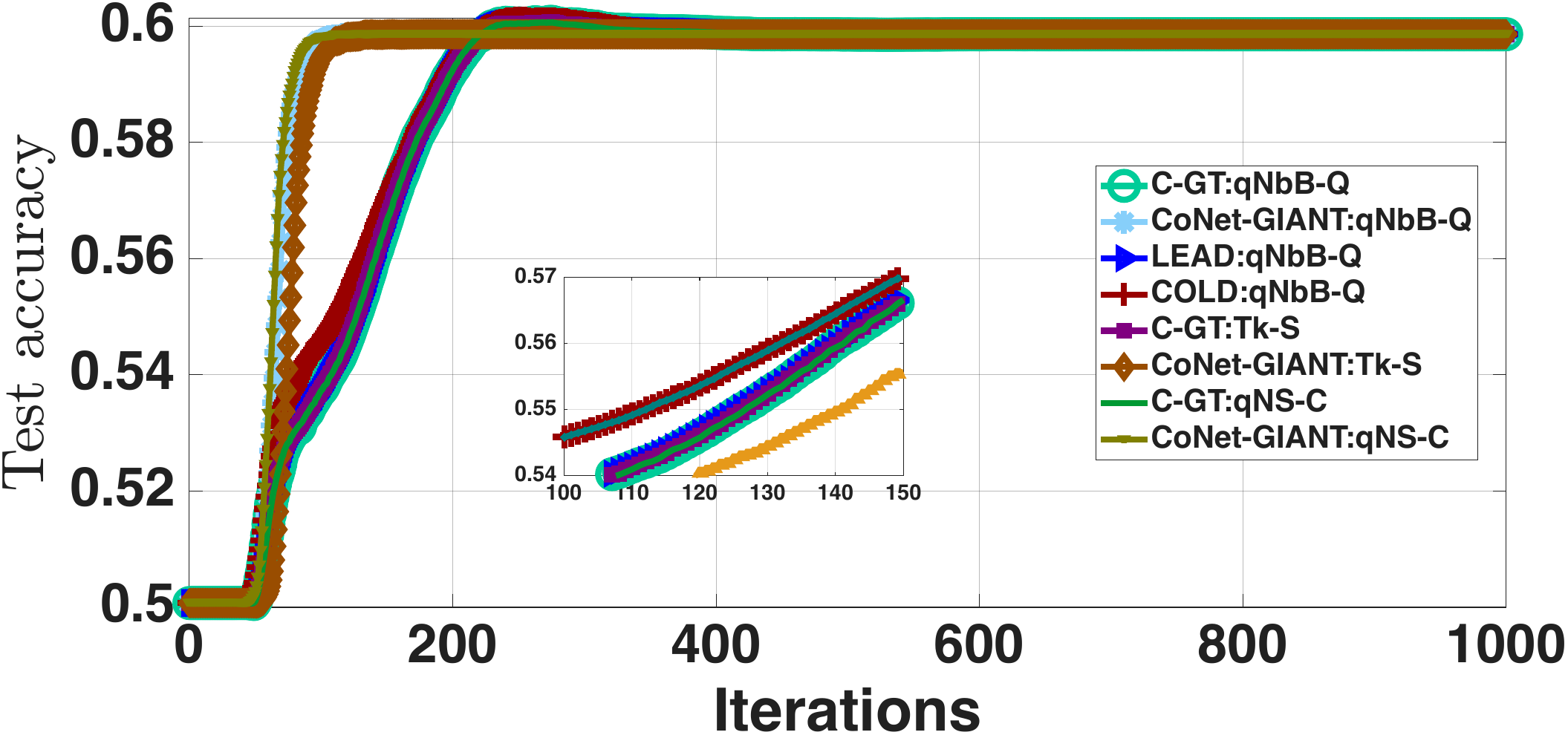}
        \caption{}
        \label{subfig:fc_acc_cov_ex}
    \end{subfigure}
    \caption{Comparison of \(\algoname\) with \(\cgt\), \(\lead\), \(\cold\), and \(\comphess\), for the expander network.}
    \label{fig:covtype_qnb_ex}
\end{figure}

\section{Conclusion}
\label{sec:conclusion}
This article presents a communication-efficient, fully decentralized Newton-type optimization algorithm with compression. We established a linear convergence rate comparable to its uncompressed counterparts, and superior to gradient-descent-based distributed optimization algorithms. Extensive experiments are performed on the synthetic and CovType datasets to demonstrate the effectiveness of the algorithm.
\bibliographystyle{IEEEtran}
\bibliography{refs}

\appendices
\onecolumn
\section{Proofs}
\label{sec:proofs}
The appendix is organized as follows: We start with several technical lemmas and results (see Lemmas \ref{lem:techinical lemmas} and \ref{lem:portmanteuau 1}, and  below \S~\ref{subappen:technical}) that are required for the proof of Lemma \ref{lem: key lemma} and Theorem \ref{thm:key thm}, which are included in \S~\ref{subappen:main proof}.

\subsection{Technical results}\label{subappen:technical}
\begin{lemma}
    \label{lem:techinical lemmas}
    Consider Algorithm \ref{alg:sec_ord_comp} along with its associated data and notations established in \S\ref{sec:algo}. Suppose that Assumptions \ref{assum:on graph} and \ref{assum:Standard assumptions} hold. Let \(\Wght\) be the consensus weight matrix.
    \begin{enumerate}[leftmargin=*,label = (\ref{lem:techinical lemmas}-\alph*)]
        \item \label{it:error_comp_1} For every \(t \ge 0\), we have the following equalities:
        \begin{align*}
            \wedvm(t) = \Wght \edvm(t) \text{ and } \wegtm(t) = \Wght \egtm(t).
        \end{align*}

        \item \label{it:error_comp_2} Let \((\sigalg_t)_{t \in \Nz}\) denotes a sequence of \(\sigma\)-algebras generated by the history 
        \[\aset[]{\dmat(0), \gtmat(0),\dmat(1), \gtmat(1), \ldots, \dmat(t), \gtmat(t)},\]
        i.e., \(\sigalg_t = \sigma \bigl( \dmat(0), \gtmat(0),\dmat(1), \gtmat(1), \ldots, \dmat(t), \gtmat(t)\bigr)\).
        Consider the update rule in Line \(3\) of Algorithm \ref{alg:sec_ord_comp}. Let \(\Cgen>0\) be the constant in Definition \ref{def:General Compression Operator}. Then for every iteration \(t \ge 0\)
        \begin{equation*}
            \EE \cexpecof[\bigg]{\norm{\dmat(t) - \edvm(t)}^2 \given \sigalg_t} \le \Cgen \norm{\dmat(t) - \auxdv(t)}^2 \text{ and }\EE \cexpecof[\bigg]{\norm{\gtmat(t) - \egtm(t)}^2 \given  \sigalg_t} \le \Cgen \norm{\gtmat(t) - \auxgt(t)}^2.
        \end{equation*} \hfill \qedsymbol
    \end{enumerate}
\end{lemma}
\begin{proof}
    Let us prove the first part. From Algorithm \ref{alg:sec_ord_comp}, we observe that \(\wauxdv(0) = \Wght \auxdv(0)\) and \(\wauxgt(0) = \Wght \auxgt(0)\). Using the expression of Line \(7\), Line \(8\) of Algorithm \ref{alg:sec_ord_comp} simplifies to
\begin{align}\label{eq:iter step prel}
    \wedvm(0) = \wauxdv(0) + \Wght \encode{\dvar}(0)
    &= \Wght \auxgt(0) + \Wght \encode{\dvar}(0) = \Wght \bigl( \underbrace{\auxgt(0) + \encode{\dvar}(0)}_{= \edvm(0)} \bigr) = \Wght \edvm(0).
\end{align}
Immediately, from Line \(9\) we see that
\begin{align}
    \label{eq:inter step}
    \wauxdv(1) &= (1 - \condv)\wauxdv(0) + \condv \wedvm(0) = \Wght \Bigl((1 - \condv)\auxdv(0) + \condv \edvm(0) \Bigr) = \Wght \auxdv(1)
\end{align}
Iterating \eqref{eq:iter step prel} and \eqref{eq:inter step} continuously for each iterations \(t >0\), we obtain
\(
    \wauxdv(t) = \Wght 
    \auxdv(t) \text{ and } \wedvm(t) = \Wght \edvm(t).
\)
Using similar arguments, we obtain
\(
    \wauxgt(t) = \Wght \auxgt(t) \text{ and } \wegtm(t) = \Wght \egtm(t).
\)
This concludes the proof of the first part.

For the second part, let us consider the \(\sigma\)-algebra \(\sigalg_t\) generated by the history until time \(t\),
\(\aset[]{\dmat(0), \gtmat(0),\dmat(1), \gtmat(1), \ldots, \dmat(t), \gtmat(t)}\).
Observe that 
\begin{align*}
    \EE \cexpecof[\bigg]{\norm{ \dmat(t) - \edvm(t)}^2 \given \sigalg_t} &= \EE \cexpecof[\bigg]{\norm{\dmat(t) - \gcompop\bigl( \dmat(t) - \auxdv(t)\bigr) - \auxdv(t)}^2\given \sigalg_t} \\
    & = \EE \cexpecof[\bigg]{\norm{\dmat(t)  - \auxdv(t) - \gcompop\bigl( \dmat(t) - \auxdv(t)\bigr)}^2\given \sigalg_t} \\ 
    & \stackrel{(\star)}{\le} \Cgen \norm{\dmat(t)  - \auxdv(t)}^2,
\end{align*}
where \((\star)\) follows from Definition \ref{def:General Compression Operator}. A similar conclusion can be drawn for the gradient tracker \(\gtmat\), and we recover 
\(\EE \cexpecof[\bigg]{\norm{ \gtmat(t) - \egtm(t)}^2 \given \sigalg_t} \le  \Cgen \norm{\gtmat(t) - \auxgt(t)}^2.\)
The proof is now complete.
\end{proof}

We next present several properties of the variables of interest: the gradient-tracking variable \(\gtmat\) and the decision variable \(\dmat\).
\begin{lemma}
    \label{lem:portmanteuau 1}
    Suppose that the hypothesis of Lemma \ref{lem:techinical lemmas} holds. Then the following assertions hold:
    \begin{enumerate}[leftmargin=*,label = (\ref{lem:portmanteuau 1}-\alph*)]
        \item \label{it:port_1}If the initial condition \(\gtmat(0)\) is chosen such that \(\gtmat(0) = \grd \Objfunc(\dmat(0))\), then for every \(t\ge 0\) we have \(\avggtmat(t) = \frac{1}{\agents} \ones^{\top} \grd \Objfunc(\dmat(t))\).

        \item \label{it:port_3} Let \(L\) be the Lipschitz constant in Assumption \ref{assum:Standard assumptions}. For every \(t \ge 0\), the inequality
        \[
            \norm{\avggtmat(t)^{\top} - \grd \objfunc\bigl( \avgdmat(t)^{\top}\bigr) } \le \frac{\lips}{\sqrt{\agents}} \norm{\dmat(t) - \ones \avgdmat(t)}
        \]
        holds.
        
        \item \label{it:port_2} For every \(t \ge 0\), we get
        \(\avgdmat(t+1) = \avgdmat(t) - \stsz \frac{1}{\agents}\ones^{\top} \cdmat(t).\)

        \item \label{it:port_4} Recall that \(\strconv>0\) is the constant associated with strong convexity of the sequence of functions \((\objfunc_i)_{i=1}^{\agents}\) in Assumption \ref{assum:Standard assumptions}. For each \(t \ge 0\), the inequality \(\norm{\cdmat(t)  - \ones \avgcdmat(t)}^2 \le \frac{1}{\strconv^2} \norm{\gtmat(t)}^2\) holds.

        \item \label{it:port_5} The following inequality is valid for each \(t\):
        \begin{align*}
            \norm{\grd \av{\Objfunc} \big(\ones \avgdmat(t) \big) - \grd \objfunc\big(\dvar^{\ast}\big)} \le \lips \norm{\avgdmat(t)^{\top} - \dvar^{\ast}}.
        \end{align*}

        \item \label{it:port_6} Define the quantity \(\quanto \Let \norm{\identity{\agents} - \Wght}.\)
        For each iteration \(t\), 
        \begin{align*}
            &\norm{\grd \Objfunc\big(\dmat(t+1)\big) - \Objfunc\big(\dmat(t)\big)}^2 \le 3 \lips^2 \constsz^2  \quanto^2 \norm{\dmat(t) - \edvm(t)}^2 + 3 \lips^2 \constsz^2  \quanto^2 \norm{\dmat(t) - \ones \avgdmat(t)}^2 + \frac{3\lips^2 \stsz^2}{\strconv^2} \norm{\gtmat(t)}^2.
        \end{align*}
    \end{enumerate}
\end{lemma}
\begin{proof}
    For the first part, we multiply both sides of the update equation \eqref{eq:update rule} by \(\frac{1}{\agents}\ones^{\top}\), yielding
    \begin{align*}
        \avggtmat(t+1) &= \frac{1}{\agents}\ones^{\top} \Bigl(\gtmat(t) - \constsz \bigl(\egtm(t) - \wegtm(t)\bigr) + \grd \Objfunc\bigl(\dmat(t+1)\bigr) - \grd \Objfunc\bigl(\dmat(t)\bigr) \Bigr) \\
        & = \frac{1}{\agents} \ones^{\top} \Bigl(\gtmat(t) + \grd \Objfunc\bigl(\dmat(t+1)\bigr) - \grd \Objfunc\bigl(\dmat(t)\bigr) \Bigr) - \frac{\constsz }{\agents} \ones^{\top} \bigl(\identity{\agents} - \Wght\bigr)\egtm(t).
    \end{align*}
    Note that \(\ones \in \mathrm{Ker}(\identity{\agents} - \Wght)\), implying that \(\ones^{\top} (\identity{\agents} - \Wght) = \ones^{\top} - \ones^{\top}\Wght = \ones^{\top} - \ones^{\top} = 0\). Therefore, \(\frac{\constsz }{\agents} \ones^{\top} \bigl(\identity{\agents} - \Wght\bigr)\egtm(t)=0\). Using the update equation \eqref{eq:update rule} iteratively, we simplify
    \begin{align*}
        & \frac{1}{\agents} \ones^{\top} \Bigl(\gtmat(t) + \grd \Objfunc\bigl(\dmat(t+1)\bigr) - \grd \Objfunc\bigl(\dmat(t)\bigr) \Bigr) \\
        & = \frac{1}{\agents} \ones^{\top}  \Bigl(\gtmat(0) - \grd \Objfunc\bigl( \dmat(0)\bigr) \Bigr) + \frac{1}{\agents} \ones^{\top} \grd \Objfunc\bigl(\dmat(t+1)\bigr) \\
        & = \frac{1}{\agents} \ones^{\top} \grd \Objfunc\bigl(\dmat(t+1)\bigr) \text{ upon choosing } \gtmat(0) = \grd \Objfunc\bigl( \dmat(0)  \bigr).  
    \end{align*}
    This concludes the first part. 

    For the second part, using Lemma \ref{it:port_1}, we write
    \begin{align*}
        \norm{\avggtmat(t) - \grd \objfunc\bigl( \avgdmat(t)^{\top} \bigr)} &= \norm{\grd \av{\Objfunc}(\dmat(t)) - \grd \objfunc\bigl( \avgdmat(t)^{\top} \bigr)}\\
        & \le \norm{\frac{1}{\agents} \sum_{i=1}^{\agents}\Bigl( \grd \objfunc_i ( \dvar_i(t)) - \grd \objfunc_i\bigl( \avgdmat(t)^{\top} \bigr)\Bigr)} \\
        & \le \frac{1}{\agents} \sum_{i=1}^{\agents} \norm{\grd \objfunc_i ( \dvar_i(t)) - \grd \objfunc_i\bigl( \avgdmat(t)^{\top}\bigr)} \\ 
        & \le \frac{\lips}{\agents} \sum_{i=1}^{\agents} \norm{\dvar_i(t) - \avgdmat(t)^{\top}} \le \sqrt{ \frac{\lips^2}{\agents} \sum_{i=1}^{\agents} \norm{\dvar_i(t) - \avgdmat(t)^{\top}}^2 } = \frac{\lips}{\sqrt{\agents}} \norm{\dmat(t) - \ones \avgdmat(t)},
    \end{align*}
    completing the proof for the second part.

    The proof for the third part is straightforward, and follows by multiplying both sides of the update rule for the decision variable in \eqref{eq:update rule}.

    Let us now focus on proving the fourth part. For each iteration \(t\), expanding the square in \(\norm{\cdmat(t) - \ones \avgcdmat(t)}^2 \) and simplifying the resultant expressions, we get
    \begin{align}\label{eq:random}
        \norm{\cdmat(t) - \ones \avgcdmat(t)}^2 
        &= \norm{\cdmat(t)}^2 + \agents \norm{\frac{1}{\agents} \ones^{\top} \cdmat(t)}^2 \nn\\ & \qquad \qquad - 2 \sum_{i=1}^{\agents} \inprod{\hess \objfunc_i \big(\dvar_i(t)\big){\inverse} \gtvar_i(t)}{\frac{1}{\agents} \ones^{\top} \cdmat(t)}\nn\\
        & = \norm{\cdmat(t)}^2 + \agents \norm{\frac{1}{\agents} \ones^{\top} \cdmat(t)}^2  - 2 \agents \underbrace{\inprod{\frac{1}{\agents} \ones^{\top} \cdmat(t)}{\frac{1}{\agents} \ones^{\top} \cdmat(t)}}_{\teL \norm{\frac{1}{\agents} \ones^{\top} \cdmat(t)}^2} \nn\\
        & = \norm{\cdmat(t)}^2 - \agents \norm{\frac{1}{\agents} \ones^{\top} \cdmat(t)}^2\nn\\
        & \le \norm{\cdmat(t)}^2 = \sum_{i=1}^{\agents} \norm{\hess \objfunc_i \big(\dvar_i(t) \big){\inverse} \gtvar_i(t)}^2\nn\\
        & \le \frac{1}{\strconv^2} \norm{\gtmat(t)}^2 \quad \text{(from \eqref{eq:strongly_convex_Lips})},
    \end{align}
    concluding the proof.

    The proof for the fifth part is simple. Expanding the terms, we have the following set of inequalities:
    \begin{align*}
        \norm{\grd \av{\Objfunc} \big(\ones \avgdmat(t) \big) - \grd \objfunc\big(\dvar^{\ast}\big)} &\le \norm{\frac{1}{\agents} \sum_{i=1}^{\agents} \Big(\grd \objfunc_i\big(\avgdmat(t)^{\top} \big) - \grd \objfunc_i(\dvar^{\ast}) \Big)}\\
        & \le \frac{1}{\agents}\sum_{i=1}^{\agents} \norm{\grd \objfunc_i\big(\avgdmat(t)^{\top} \big) - \grd \objfunc_i(\dvar^{\ast})} \\
        & \stackrel{(\ddagger)}{\le} \frac{\lips}{\agents}\sum_{i=1}^{\agents} \norm{\avgdmat(t)^{\top} - \dvar^{\ast}} = \frac{\lips}{\agents} \agents \norm{\avgdmat(t)^{\top} - \dvar^{\ast}}\\
        & = \lips \norm{\avgdmat(t)^{\top} - \dvar^{\ast}},
    \end{align*}
    where \((\ddagger)\) follows from Assumption \ref{assum:Standard assumptions}, therefore, proving the assertion.

    We now focus on part \ref{it:port_6}. Using the fact that for each \(i \in \aset[]{1,2,\ldots,\agents}\), \(\objfunc_i(\cdot)\) is \(\lips\)-smooth from Assumption \ref{assum:Standard assumptions},
    \begin{align*}
        \norm{\grd \Objfunc\big(\dmat(t+1)\big) - \Objfunc\big(\dmat(t)\big)}^2 
        &  \le \sum_{i=1}^{\agents} \norm{\grd \objfunc_i(\dvar_i(t+1)) - \objfunc_i(\dvar_i(t))}^2 \\
        & \le \lips^2 \norm{\dmat(t+1) - \dmat(t)}^2.
    \end{align*}
    Let us now analyze the term \(\norm{\dmat(t+1) - \dmat(t)}^2\). Observe that \((\identity{\agents} - \Wght)\ones \avgdmat(t) = \ones \avgdmat(t) - \Wght \ones \avgdmat(t) =  \ones \avgdmat(t) - \ones \avgdmat(t) = 0\). Using this and from \eqref{eq:update rule}, the above term can be simplified as
    \begin{align}\label{eq:crucial_step}
        &\norm{\dmat(t+1) - \dmat(t)}^2 \nn\\
        & = \left\lVert\constsz(\identity{\agents} - \Wght)\dmat(t) - \constsz(\identity{\agents} - \Wght)\dmat(t) - \constsz(\identity{\agents} - \Wght)\edvm(t) \right. \nn\\& \hspace{1cm}+ \left. \constsz(\identity{\agents} - \Wght)\ones \avgdmat(t) - \stsz \cdmat(t)\right\rVert^2\nn\\
        & = \left\lVert \constsz(\identity{\agents} - \Wght)\big(\dmat(t) - \edvm(t)\big) - \constsz(\identity{\agents} - \Wght)\big(\dmat(t) - \ones \avgdmat(t)\big) - \stsz \cdmat(t))\right\rVert^2\nn\\
        & \stackrel{(a)}{\le} 3 \constsz^2 \quanto^2 \norm{\dmat(t) - \edvm(t)}^2 + 3 \constsz^2 \quanto^2 \norm{\dmat(t) - \ones \avgdmat(t)}^2 + 3 \stsz^2 \norm{\cdmat(t)}^2\nn\\
        & \stackrel{(b)}{\le}  3 \constsz^2 \quanto^2 \norm{\dmat(t) - \edvm(t)}^2 + 3 \constsz^2 \quanto^2 \norm{\dmat(t) - \ones \avgdmat(t)}^2 + \frac{3 \stsz^2}{\strconv^2} \norm{\gtmat(t)}^2. 
    \end{align}
    Here \((a)\) is implied from \eqref{eq:aux_step2} and \((b)\) follows from \eqref{eq:random}. Therefore,
    \begin{align*}
        &\norm{\grd \Objfunc\big(\dmat(t+1)\big) - \Objfunc\big(\dmat(t)\big)}^2 \\
        & \le \lips^2 \norm{\dmat(t+1) - \dmat(t)}^2\\
        & \le 3 \lips^2 \constsz^2 \quanto^2 \norm{\dmat(t) - \edvm(t)}^2 + 3 \lips^2\constsz^2 \quanto^2 \norm{\dmat(t) - \ones \avgdmat(t)}^2 + \frac{3 \lips^2 \stsz^2}{\strconv^2} \norm{\gtmat(t)}^2.
    \end{align*}
    This completes the proof of part \eqref{it:port_6}.
\end{proof}

\subsection{Proof of Theorem \ref{thm:key thm}}\label{subappen:main proof}
We state the following crucial result without its proof:
\begin{proposition}
    \label{prop:aux_res_1}
    For any matrices \(A_1, A_2 \in \Rbb^{\agents \times \dimdv}\), we have 
    \begin{align*}
        \norm{A_1 + A_2}^2 \le (1+\tau')\norm{A_1}^2 + \bigg(1 + \frac{1}{\tau'} \bigg)\norm{A_2}^2,
    \end{align*}
    for any constant \(\tau' > 0\). Moreover, let \(A_3 \in \Rbb^{\agents \times \dimdv}\). Then we also have
    \begin{align*}
        &\norm{A_1 + A_2}^2 \le \ol{\tau} \norm{A_1}^2 + \frac{\ol{\tau}}{\ol{\tau} - 1}\norm{A_2}^2 \text{ and } \norm{A_1 + A_2 + A_3}^2 \le \wt{\tau} \norm{A_1}^2 + \frac{2 \wt{\tau} }{\wt{\tau} - 1} \Big(\norm{A_2}^2 + \norm{A_3}^2\Big)
    \end{align*}
    for any constant \(\ol{\tau},\, \wt{\tau} >1\). \hfill \qedsymbol
\end{proposition}
Substituting \(\ol{\tau} = 2\) and \(\wt{\tau} = 3\), the above inequalities simplifies as
\begin{align}\label{eq:aux_step2}
    \begin{cases}
        \norm{A_1 + A_2}^2 \le 2\norm{A_1}^2 + 2 \norm{A_2}^2,\\
        \norm{A_1 + A_2 + A_3}^2 \le 3 \norm{A_1}^2 + 3 \norm{A_2}^2 + 3\norm{A_3}^2.
    \end{cases}
\end{align}
We proceed to the proof of Lemma \ref{lem: key lemma}.
\begin{proof}[Proof of Lemma \ref{lem: key lemma}]
    The idea, motivated by \cite{ref:GQ-NL-17}, is to bound each of the errors in \(\err(\cdot)\) in terms of their previous values. Let us start by analyzing the optimization error \(\opterr(\cdot)\). Recall that \(\sigalg_t = \sigma \bigl( \dmat(0), \gtmat(0),\dmat(1), \gtmat(1), \ldots, \dmat(t), \gtmat(t)\bigr)\), and the notation established in \S \ref{sec:algo}. 

    \subsection*{Step \(1\): Analysis of the optimality error}
    Observe that 
    \begin{align*}
        \av{\cdmat}(t) & = \frac{1}{\agents}\ones^{\top} \cdmat(t) = \frac{1}{\agents} \sum_{i=1}^{\agents} \biggl(\hess \objfunc_i(\dvar_i(t)){\inverse} \Bigl(\gtvar_i(t) - \grd \objfunc \bigl(\avgdmat(t)^{\top} \bigr) + \grd \objfunc \bigl(\avgdmat(t)^{\top}\bigr) \Bigr)\biggr)^{\top}\\
        & = \frac{1}{\agents} \sum_{i=1}^{\agents} \biggl(\hess \objfunc_i(\dvar_i(t)){\inverse} \Bigl(\gtvar_i(t) - \grd \objfunc \bigl(\avgdmat(t)^{\top} \bigr) \Bigr)\biggr)^{\top} \\ & \hspace{4cm} + \frac{1}{\agents} \sum_{i=1}^{\agents} \Bigl(\hess \objfunc_i(\dvar_i(t)){\inverse}  \grd \objfunc \bigl(\avgdmat(t)^{\top}\bigr) \Bigr)^{\top}.
    \end{align*}
    Substituting the above in \eqref{eq:update rule} and invoking Lemma \ref{it:port_2}, we get
    \begin{align*}
        \avgdmat(t+1)^{\top}  - \dvar^{\ast}&= \avgdmat(t)^{\top} - \stsz \frac{1}{\agents} \cdmat(t)^{\top}\ones - \dvar^{\ast} \\
        & = \avgdmat(t)^{\top} - \stsz \frac{1}{\agents} \sum_{i=1}^{\agents} \hess \objfunc_i(\dvar_i(t)){\inverse}  \grd \objfunc \bigl(\avgdmat(t)^{\top}\bigr) - \dvar^{\ast} \\ & \hspace{1.5cm} - \stsz \frac{1}{\agents}\sum_{i=1}^{\agents} \hess \objfunc_i(\dvar_i(t)){\inverse} \Bigl(\gtvar_i(t) - \grd \objfunc \bigl(\avgdmat(t)^{\top} \bigr) \Bigr).
    \end{align*}
    Now, consider the term
    \begin{align*}
        &\norm{\avgdmat(t)^{\top} - \stsz \frac{1}{\agents} \sum_{i=1}^{\agents} \hess \objfunc_i(\dvar_i(t)){\inverse}  \grd \objfunc \bigl(\avgdmat(t)^{\top}\bigr) - \dvar^{\ast} }\\
        & = \norm{\avgdmat(t)^{\top} - \dvar^{\ast} - \stsz \frac{1}{\agents} \sum_{i=1}^{\agents} \hess \objfunc_i(\dvar_i(t)){\inverse} \Bigl( \grd \objfunc \bigl(\avgdmat(t)^{\top}\bigr)  - \grd \objfunc(\dvar^{\ast}) \Bigr) }\\
        & = \norm{\avgdmat(t)^{\top} - \dvar^{\ast} - \stsz \frac{1}{\agents} \sum_{i=1}^{\agents} \hess \objfunc_i(\dvar_i(t)){\inverse} \int_0^1 \hess \objfunc\Bigl(\dvar^{\ast} + \varsigma \bigl(\avgdmat(t)^{\top} - \dvar^{\ast} \bigr)\Bigr) \bigl(\avgdmat(t)^{\top} - \dvar^{\ast} \bigr) \odif{\varsigma} } \\
        & = \norm{\bigg(\identity{\agents} - \stsz \frac{1}{\agents} \sum_{i=1}^{\agents} \hess \objfunc_i(\dvar_i(t)){\inverse} \int_0^1 \hess \objfunc\Bigl(\dvar^{\ast} + \varsigma \bigl(\avgdmat(t)^{\top} - \dvar^{\ast} \bigr)\Bigr)\odif{\varsigma}\bigg)\bigl(\avgdmat(t)^{\top} - \dvar^{\ast} \bigr)}\\
        & \le \norm{\bigg(\identity{\agents} - \stsz \frac{1}{\agents} \sum_{i=1}^{\agents} \hess \objfunc_i(\dvar_i(t)){\inverse} \int_0^1 \hess \objfunc\Bigl(\dvar^{\ast} + \varsigma \bigl(\avgdmat(t)^{\top} - \dvar^{\ast} \bigr)\Bigr)\odif{\varsigma} \bigg)} \norm{\bigl(\avgdmat(t)^{\top} - \dvar^{\ast} \bigr)},
    \end{align*}
    where the last inequality follows from the sub-multiplicativity of the norm. From \eqref{eq:strongly_convex_Lips} and from the fact that \(\stsz \le \frac{2\lips}{3 \strconv} \le \frac{\lips}{\strconv}\) such that \(1 - \stsz\frac{\strconv}{\lips} \ge 0\) is ensured
    , we write
    \begin{align*}
        &\norm{\bigg(\identity{\agents} - \stsz \frac{1}{\agents} \sum_{i=1}^{\agents} \hess \objfunc_i(\dvar_i(t)){\inverse} \int_0^1 \hess \objfunc\Bigl(\dvar^{\ast} + \varsigma \bigl(\avgdmat(t)^{\top} - \dvar^{\ast} \bigr)\Bigr)\odif{\varsigma} \bigg)}\\
        & \le \norm{\bigg(1 - \stsz \frac{1}{\agents \lips} \strconv \agents \bigg)\identity{\agents}} = \norm{\bigg(1 - \stsz \frac{\strconv}{\lips} \bigg)\identity{\agents}},
    \end{align*}
    and plugging it back, we recover
    \begin{align}\label{eq:arb step}
        &\norm{\avgdmat(t)^{\top} - \stsz \frac{1}{\agents} \sum_{i=1}^{\agents} \hess \objfunc_i(\dvar_i(t)){\inverse}  \grd \objfunc \bigl(\avgdmat(t)^{\top}\bigr) - \dvar^{\ast} }\nn \\
        & \le \norm{\bigg(1 - \stsz \frac{\strconv}{\lips} \bigg)\identity{\agents}} \norm{\bigl(\avgdmat(t)^{\top} - \dvar^{\ast} \bigr)}\nn \\
        & \le \bigg(1 - \stsz \frac{\strconv}{\lips} \bigg) \norm{\bigl(\avgdmat(t)^{\top} - \dvar^{\ast} \bigr)}.
    \end{align}
    We now focus on the term \(T \Let \norm{\frac{1}{\agents}\sum_{i=1}^{\agents} \hess \objfunc_i(\dvar_i(t)){\inverse} \Bigl( \grd \objfunc \bigl(\avgdmat(t)^{\top} \bigr) - \gtvar_i(t)  \Bigr)}^2\),
    which can be written as
    \begin{align*}
        T &= \norm{\frac{1}{\agents}\sum_{i=1}^{\agents} \hess \objfunc_i(\dvar_i(t)){\inverse} \Bigl( \grd \objfunc \bigl(\avgdmat(t)^{\top} \bigr) - \grd \ol{\Objfunc}\big( \dmat(t)\big) + \grd \ol{\Objfunc}\big( \dmat(t)\big) - \gtvar_i(t)  \Bigr)}^2  \\
        & \le 2\norm{\frac{1}{\agents}\sum_{i=1}^{\agents} \hess \objfunc_i(\dvar_i(t)){\inverse} \Bigl( \grd \objfunc \bigl(\avgdmat(t)^{\top} \bigr) - \grd \ol{\Objfunc}\big( \dmat(t)\big) \Bigr)}^2 \\ & \hspace{1.5cm}+ 2\norm{\frac{1}{\agents}\sum_{i=1}^{\agents} \hess \objfunc_i(\dvar_i(t)){\inverse} \Bigl(  \grd \ol{\Objfunc}\big( \dmat(t)\big) - \gtvar_i(t)  \Bigr)}^2,
    \end{align*}
    where the above inequality follows from Proposition \ref{prop:aux_res_1} by substituting \(\tau'=1\). We now focus on the first term and establish an upper bound. Note that \[\norm{\frac{1}{\agents}\sum_{i=1}^{\agents} \hess \objfunc_i(\dvar_i(t)){\inverse} } \le \frac{1}{\agents}\sum_{i=1}^{\agents}\norm{\hess \objfunc_i(\dvar_i(t)){\inverse} } \le \frac{1}{\strconv}\]
    and consequently, one can write
    \begin{align*}
        I_1 &\Let \norm{\frac{1}{\agents}\sum_{i=1}^{\agents} \hess \objfunc_i(\dvar_i(t)){\inverse} \Bigl( \grd \objfunc \bigl(\avgdmat(t)^{\top} \bigr) - \grd \ol{\Objfunc}\big( \dmat(t)\big) \Bigr)}^2 \\
        & = \stsz^2 \norm{\frac{1}{\agents}\sum_{i=1}^{\agents} \hess \objfunc_i(\dvar_i(t)){\inverse} }^2 \norm{\grd \objfunc \bigl(\avgdmat(t)^{\top} \bigr) - \grd \ol{\Objfunc}\big( \dmat(t)\big)}^2\\
        & \le \frac{\stsz^2}{\strconv^2} \norm{\grd \objfunc \bigl(\avgdmat(t)^{\top} \bigr) - \grd \ol{\Objfunc}\big( \dmat(t)\big)}^2
        \le \frac{\stsz^2 \lips}{\strconv^2 \agents} \norm{\dmat(t) - \ones \ol{\dmat}(t)}^2.
    \end{align*}
    The last inequality hold because \[\norm{\grd \objfunc \bigl(\avgdmat(t)^{\top} \bigr) - \grd \ol{\Objfunc}\big( \dmat(t)\big)} \le \frac{\lips}{\agents} \sum_{i=1}^{\agents} \norm{\ol{\dmat}(t)^{\top} - \dvar_i(t)} \le \frac{\lips}{\sqrt{\agents}} \norm{\dmat(t) - \ones \ol{\dmat}(t)},\]
    and therefore, \(\norm{\grd \objfunc \bigl(\avgdmat(t)^{\top} \bigr) - \grd \ol{\Objfunc}\big( \dmat(t)\big)}^2 \le \frac{\lips}{\sqrt{\agents}} \norm{\dmat(t) - \ones \ol{\dmat}(t)}^2.\) Let us now focus on the second term \(\sqrt{I_2} \Let \norm{\frac{1}{\agents}\sum_{i=1}^{\agents} \hess \objfunc_i(\dvar_i(t)){\inverse} \Bigl(  \grd \ol{\Objfunc}\big( \dmat(t)\big) - \gtvar_i(t)  \Bigr)}.\) Observe that
    \begin{align*}
        \sqrt{I_2} \le \frac{\stsz}{\agents} \sum_{i=1}^{\agents} \norm{\hess \objfunc_i(\dvar_i(t)){\inverse} \big(\grd \ol{\Objfunc}(\dmat(t)) - \gtvar_i(t)\big)} &\le \frac{\stsz}{\strconv \agents}\sum_{i=1}^{\agents} \norm{\grd \ol{\Objfunc}(\dmat(t)) - \gtvar_i(t)}\\
        & \le \frac{\stsz}{\strconv \sqrt{\agents}} \norm{\gtmat(t) - \ones \grd \ol{\Objfunc}(\dmat(t))}.
    \end{align*}
    Therefore, \(I_2^2 \le \frac{\stsz^2}{\strconv^2 \agents} \norm{\gtmat(t) - \ones \grd \ol{\Objfunc}(\dmat(t))}^2\), and finally, we obtain
    \begin{align}\label{eq:aux_stepp}
        T &\le 2I_1 + 2I_2
        \le \frac{2\stsz^2 \lips}{\strconv^2 \agents} \norm{\dmat(t) - \ones \ol{\dmat}(t)}^2 + \frac{2\stsz^2}{\strconv^2 \agents} \norm{\gtmat(t) - \ones \grd \ol{\Objfunc}(\dmat(t))}^2.
    \end{align}
    Let us now consider the optimization error
    \begin{align*}
        &\EE \cexpecof[\Big]{\norm{\avgdmat(t+1)^{\top}  - \dvar^{\ast}}^2 \given \sigalg_t} = \left \lVert\avgdmat(t)^{\top} - \stsz \frac{1}{\agents} \sum_{i=1}^{\agents} \hess \objfunc_i(\dvar_i(t)){\inverse}  \grd \objfunc \bigl(\avgdmat(t)^{\top}\bigr) - \dvar^{\ast}  \right. \\ & \hspace{1cm} \left. +\stsz \frac{1}{\agents}\sum_{i=1}^{\agents} \hess \objfunc_i(\dvar_i(t)){\inverse} \Bigl( \grd \objfunc \bigl(\avgdmat(t)^{\top} \bigr) - \gtvar_i(t)  \Bigr) \right \rVert^2 \\
        & \le (1 + \tau_1) \left \lVert\avgdmat(t)^{\top} - \stsz \frac{1}{\agents} \sum_{i=1}^{\agents} \hess \objfunc_i(\dvar_i(t)){\inverse}  \grd \objfunc \bigl(\avgdmat(t)^{\top}\bigr) - \dvar^{\ast}  \right \rVert^2 \\ & \quad+ \biggl( 1 + \frac{1}{\tau_1}\biggr) \norm{\stsz \frac{1}{\agents}\sum_{i=1}^{\agents} \hess \objfunc_i(\dvar_i(t)){\inverse} \Bigl( \grd \objfunc \bigl(\avgdmat(t)^{\top} \bigr) - \gtvar_i(t)  \Bigr)}^2\\
        & \le (1 + \tau_1) \bigg(1 - \stsz \frac{\strconv}{\lips} \bigg)^2 \norm{\bigl(\avgdmat(t)^{\top} - \dvar^{\ast} \bigr)}^2 + \biggl( 1 + \frac{1}{\tau_1}\biggr) \frac{2\stsz^2 \lips^2}{\strconv^2 \agents} \norm{ \avgdmat(t)^{\top}  - \ones \avgdmat(t) }^2 \\&\hspace{1cm} + \frac{2\stsz^2}{\strconv^2 \agents} \norm{\gtmat(t) - \ones \grd \ol{\Objfunc}(\dmat(t))}^2 
    \end{align*}
    Choose \(\tau_1 = \frac{\stsz \strconv}{2 \lips}\), and observe that 
    \((1 + \tau_1) \big(1 - \stsz \frac{\strconv}{\lips} \big)^2 = \big(1 + \frac{\stsz \strconv}{2 \lips}\big) \big(1 - \stsz \frac{\strconv}{\lips} \big)^2 = 1 - \frac{3 \stsz \strconv}{2 \lips} + \frac{\stsz^3 \strconv^3}{2 \lips^3}\), \(\bigl( 1 + \frac{1}{\tau_1}\bigr) \frac{\stsz^2 \lips^2}{\strconv^2 \agents} = \frac{\stsz^2 \lips^2}{\strconv^2 \agents} + \frac{2 \stsz^2 \lips^3}{\stsz \strconv^3 \agents}\), and \(\bigl( 1 + \frac{1}{\tau_1}\bigr) \frac{\stsz^2 }{\strconv^2 \agents} = \frac{\stsz^2 }{\strconv^2 \agents} + \frac{2 \stsz^2 \lips}{\stsz \strconv^3 \agents}\).
    Combining everything, we get the following final upper bound on the optimality error
    \begin{align}\label{eq:opt err}
        &\EE \expecof[\Big]{\norm{\avgdmat(t+1)^{\top}  - \dvar^{\ast}}^2}
        \le \bigg( 1 - \frac{3 \stsz \strconv}{2 \lips} + \frac{\stsz^3 \strconv^3}{2 \lips^3}\bigg) \norm{\bigl(\avgdmat(t)^{\top} - \dvar^{\ast} \bigr)}^2 \nn \\& \hspace{4.5cm} + \bigg(\frac{\stsz^2 \lips^2}{\strconv^2 \agents} + \frac{2 \stsz^2 \lips^3}{\stsz \strconv^3 \agents} \bigg) \norm{ \avgdmat(t)^{\top}  - \ones \avgdmat(t) }^2 \nn \\&
        \hspace{2cm} \bigg(\frac{\stsz^2 }{\strconv^2 \agents} + \frac{2 \stsz^2 \lips}{\stsz \strconv^3 \agents} \bigg) \norm{\gtmat(t) - \ones \grd \ol{\Objfunc}(\dmat(t))}^2.
    \end{align}
    
    \subsection*{Step \(2\): Analyzing the consensus error} Invoking Lemma \ref{it:error_comp_1}, we get the expression
    \begin{align*}
        &\dmat(t) - \constsz \big(\edvm(t) - \wedvm(t) \big) - \ones \avgdmat(t)\\
        & =  \dmat(t) - \constsz \big(\identity{\agents} - \Wght \big) \edvm(t) - \ones \avgdmat(t) + \constsz \big(\identity{\agents} - \Wght\big) \dmat(t) - \constsz \big(\identity{\agents} - \Wght\big) \dmat(t) \\
        & = \constsz \big(\identity{\agents} - \Wght\big) \big(\dmat(t) - \edvm(t) \big) + \underbrace{\big((1-\constsz)\identity{\agents} + \constsz \Wght \big)}_{\teL \wt{\Wght}} \dmat(t) - \ones \avgdmat(t)\\
        & = \constsz \big(\identity{\agents} - \Wght\big) \big(\dmat(t) - \edvm(t) \big) + \big(\wt{\Wght} \dmat(t) - \ones \avgdmat(t)\big).
    \end{align*}
    From the subsequent discussion after Assumption \ref{assum:on graph}, let us now simplify the consensus error
    \begin{align*}
        &\EE \cexpecof[\Big]{\norm{\dmat(t+1) - \ones \avgdmat(t)}^2 \given \sigalg_t}\\
        & = \EE \cexpecof[\bigg]{\norm{\big(\wt{\Wght} \dmat(t) - \ones \avgdmat(t)\big) + \constsz \big(\identity{\agents} - \Wght\big) \big(\dmat(t) - \edvm(t) \big) + \stsz \big( \ones \avgcdmat(t) - \cdmat(t)\big)}^2 \given \sigalg_t }\\
        & \stackrel{(\star)}{\le} \wt{\specnorm}^2 (1 + \tau_2) \EE \cexpecof[\Big]{\norm{\big(\dmat(t) - \ones \avgdmat(t)\big)}^2 \given \sigalg_t} + 2\stsz^2 \bigg(1 + \frac{1}{\tau_2}\bigg)  \EE \cexpecof[\Big]{\norm{ \cdmat(t) - \ones \avgcdmat(t)}^2 \given \sigalg_t} \\ & \hspace{2cm} + 2\bigg(1 + \frac{1}{\tau_2}\bigg) \constsz^2 \underbrace{\norm{\identity{\agents} - \Wght}^2}_{= \quanto} \EE \cexpecof[\Big]{\norm{\dmat(t) - \edvm (t)}^2\given \sigalg_t},
    \end{align*}
    where \((\star)\) follows from \eqref{eq:aux_step2} in Proposition
    \ref{prop:aux_res_1}. Appealing to Lemma \ref{it:error_comp_2} and Lemma \ref{it:port_4}, we have
    \begin{align*}
        &\EE \cexpecof[\Big]{\norm{\dmat(t) - \edvm(t)}^2 \given \sigalg_t} \le \Cgen \norm{\dmat(t) - \auxdv(t)}^2, \text{ and }\\
        &\EE \cexpecof[\Big]{\norm{ \cdmat(t) - \ones \avgcdmat(t)}^2 \given \sigalg_t} \le \frac{1}{\strconv^2} \norm{\gtmat(t)}^2,
    \end{align*}
    Plugging them back in, we obtain
    \begin{align}\label{eq:final_con_error}
        &\EE \cexpecof[\Big]{\norm{\dmat(t+1) - \ones \avgdmat(t)}^2 \given \sigalg_t} \nn \\
        & \le \wt{\specnorm}^2 (1 + \tau_2)  \norm{\dmat(t) - \ones \avgdmat(t)}^2 + \frac{2\stsz^2}{\strconv^2} \bigg(1 + \frac{1}{\tau_2}\bigg)  \norm{\gtmat(t)}^2 \nn \\& \hspace{2cm} + 2\bigg(1 + \frac{1}{\tau_2}\bigg) \constsz^2 \Cgen \quanto^2 \norm{\dmat(t) - \auxdv(t)}^2.
    \end{align}
    Let us upper bound  he term \(\norm{\gtmat(t)}^2\). Observe that
    \begin{align}\label{eq:prefinal_}
        &\norm{\gtmat(t)}^2 = \left \lVert \Big( \gtmat(t) - \ones \grd \av{\Objfunc} \big(\dmat(t) \big) \Big) + \Big(\ones \grd \av{\Objfunc} \big(\dmat(t) \big) - \ones \grd \av{\Objfunc} \big(\ones \avgdmat(t) \big)\Big)\right.\nn \\& \hspace{2.5cm} +\left.  \Big(\ones \grd \av{\Objfunc} \big(\ones \avgdmat(t) \big) - \ones \grd \objfunc\big(\dvar^{\ast}\big)\Big)\right \rVert^2 \nn \\
        & \le (1+ \tau_3) \norm{\gtmat(t) - \ones \grd \av{\Objfunc} \big(\dmat(t) \big)}^2 + 2 \bigg(1 + \frac{1}{\tau_3} \bigg) \underbrace{\norm{\ones \grd \av{\Objfunc} \big(\dmat(t) \big) - \ones \grd \av{\Objfunc} \big(\ones \avgdmat(t) \big)}^2}_{=\text{Term }1} \nn \\ & \hspace{2.5cm} + 2 \bigg(1 + \frac{1}{\tau_3} \bigg) \underbrace{\norm{\ones \grd \av{\Objfunc} \big(\ones \avgdmat(t) \big) - \ones \grd \objfunc\big(\dvar^{\ast}\big)}^2}_{=\text{Term }2}.
    \end{align}
    Let us first focus on Term \(1\). We note that \(\grd \av{\Objfunc}\big( \ones \avgdmat(t)\big) = \frac{1}{\agents} \sum_{i=1}^{\agents} \grd \objfunc_i \big(\avgdmat(t)^{\top}\big)^{\top} = \grd \objfunc \big(\avgdmat(t)^{\top}\big)^{\top}\), and from Lemma \ref{it:port_1}, \(\avggtmat(t) = \frac{1}{\agents} \ones^{\top} \grd \Objfunc(\dmat(t)) = \grd \av{\Objfunc}\big( \dmat(t)\big)\) for each \(t\). Therefore, from Lemma \ref{it:port_3},
    \begin{align}\label{eq:term1}
        \text{Term } 1= \norm{\ones \grd \av{\Objfunc} \big(\dmat(t) \big) - \ones \grd \av{\Objfunc} \big(\ones \avgdmat(t) \big)}^2 &= \agents \norm{\avggtmat(t)^{\top} - \grd \objfunc\bigl( \avgdmat(t)^{\top}\bigr) }^2 \nn\\
        & \le \agents \frac{\lips^2}{\agents} \norm{\dmat(t) - \ones \avgdmat(t)}^2 \nn\\
        & = \lips^2 \norm{\dmat(t) - \ones \avgdmat(t)}^2.
    \end{align}
    Let us now simplify Term \(2\): To that end, appealing to Lemma \ref{it:port_5}, Term \(2\) simplifies to
    \begin{align}
        \label{eq:term2}
        \text{Term }2= \norm{\ones \grd \av{\Objfunc} \big(\ones \avgdmat(t) \big) - \ones \grd \objfunc\big(\dvar^{\ast}\big)}^2 &= \agents \norm{\grd \av{\Objfunc} \big(\ones \avgdmat(t) \big) - \grd \objfunc\big(\dvar^{\ast}\big)}^2 \nn\\
        & = \agents \lips^2 \norm{\avgdmat(t)^{\top} - \dvar^{\ast}}^2.
    \end{align}
    Substituting \eqref{eq:term1} and \eqref{eq:term2} back to \eqref{eq:prefinal_}, 
    \begin{align}\label{eq:consen_error_gtmat}
        &\norm{\gtmat(t)}^2 \le (1+ \tau_3) \norm{\gtmat(t) - \ones \grd \av{\Objfunc} \big(\dmat(t) \big)}^2 + 2 \lips^2 \bigg(1 + \frac{1}{\tau_3} \bigg) \norm{\dmat(t) - \ones \avgdmat(t)}^2 \nn \\ & \hspace{2.5cm} + 2 \agents \lips^2 \bigg(1 + \frac{1}{\tau_3} \bigg) \norm{\avgdmat(t)^{\top} - \dvar^{\ast}}^2,
    \end{align}
    which on substituting back to \eqref{eq:final_con_error} yields
    \begin{align}\label{eq:final_con_err}
        &\EE \cexpecof[\Big]{\norm{\dmat(t+1) - \ones \avgdmat(t)}^2 \given \sigalg_t} \nn\\
        & \le \Bigg(\wt{\specnorm}^2 (1 + \tau_2) + \frac{4 \lips^2 \stsz^2}{\strconv^2}\bigg( 1 + \frac{1}{\tau_2}\bigg) \bigg( 1 + \frac{1}{\tau_3}\bigg) \Bigg) \norm{\dmat(t) - \ones \avgdmat(t)}^2 \nn\\
        & \hspace{2cm}+ \frac{4 \lips^2 \stsz^2 \agents}{\strconv^2} \bigg( 1 + \frac{1}{\tau_2}\bigg) \bigg( 1 + \frac{1}{\tau_3}\bigg) \norm{\avgdmat(t)^{\top} - \dvar^{\ast}}^2 \nn\\
        & \hspace{0.5cm} + \frac{2 \stsz^2}{\strconv^2} \bigg( 1 + \frac{1}{\tau_2}\bigg) \bigg( 1 + \tau_3\bigg) \norm{\gtmat(t) - \ones \grd \av{\Objfunc} \big(\dmat(t) \big)}^2 \nn \\
        & \hspace{3cm}+ 2 \constsz^2 \Cgen \quanto^2 \bigg(1 + \frac{1}{\tau_2}\bigg)  \norm{\dmat(t) - \auxdv(t)}^2.
    \end{align}
    Choose \(\tau_2 = \frac{1 - \wt{\specnorm}^2}{2 \wt{\specnorm}^2}\) and \(\tau_3 = 1\). Then \(1+\tau_2 = 1 + \frac{1 - \wt{\specnorm}^2}{2 \wt{\specnorm}^2} = \frac{1 + \wt{\specnorm}^2}{2 \wt{\specnorm}^2}\) and since \(\wt{\specnorm} < 1\) (from the discussion preceding Assumption \ref{assum:Standard assumptions}), we compute the coefficient as follows:
    \begin{align*}
        \begin{cases}
            &1 + \frac{1}{\tau_2} = \frac{1 + \wt{\specnorm}^2}{1 - \wt{\specnorm}^2} = \frac{1}{1 - \wt{\specnorm}} \underbrace{\frac{1 + \wt{\specnorm}^2}{1 + \wt{\specnorm}}}_{< 1} \le \frac{1}{1 - \wt{\specnorm}}; \vspace{1.5mm}\\
            & \wt{\specnorm}^2 (1 + \tau_2) + \frac{4 \lips^2 \stsz^2}{\strconv^2}\bigg( 1 + \frac{1}{\tau_2}\bigg) \bigg( 1 + \frac{1}{\tau_3}\bigg)  \vspace{1.5mm} \\
            & \hspace{1cm}= \wt{\specnorm}^2 \frac{1 + \wt{\specnorm}^2}{2 \wt{\specnorm}^2} + \frac{4 \lips^2 \stsz^2}{\strconv^2} \frac{1 + \wt{\specnorm}^2}{1 - \wt{\specnorm}^2} 2 \le \frac{1 + \wt{\specnorm}^2}{2} + \frac{8 \lips^2 \stsz^2}{\strconv^2 (1 - \wt{\specnorm})};\vspace{1.5mm} \\
            & \frac{4 \lips^2 \stsz^2 \agents}{\strconv^2} \bigg( 1 + \frac{1}{\tau_2}\bigg) \bigg( 1 + \frac{1}{\tau_3}\bigg)  = \frac{4 \lips^2 \stsz^2 \agents}{\strconv^2} \frac{1 + \wt{\specnorm}^2}{1 - \wt{\specnorm}^2} 2 \le \frac{8 \lips^2 \stsz^2 \agents}{\strconv^2(1 - \wt{\specnorm})};\vspace{1.5mm} \\
            & \frac{2 \stsz^2}{\strconv^2} \bigg( 1 + \frac{1}{\tau_2}\bigg) \bigg( 1 + \tau_3\bigg) = \frac{2 \stsz^2}{\strconv^2} \frac{1 + \wt{\specnorm}^2}{1 - \wt{\specnorm}^2} 2 \le \frac{4 \stsz^2}{\strconv^2(1 - \wt{\specnorm})};\vspace{1.5mm} \\
            & 2 \constsz^2 \Cgen \quanto^2 \bigg(1 + \frac{1}{\tau_2}\bigg) \le \frac{2 \constsz^2 \Cgen \quanto^2}{1 - \wt{\specnorm}}.
        \end{cases}
    \end{align*}
     Therefore, the consensus error \eqref{eq:final_con_err} admits the final expression
     \begin{align}
         \label{eq:final_exp_con_error}
         &\EE \cexpecof[\Big]{\norm{\dmat(t+1) - \ones \avgdmat(t)}^2 \given \sigalg_t} \nn\\
         & \le \Bigg(\frac{1 + \wt{\specnorm}^2}{2} + \frac{8 \lips^2 \stsz^2}{\strconv^2 (1 - \wt{\specnorm})} \Bigg) \norm{\dmat(t) - \ones \avgdmat(t)}^2 
         + \frac{8 \lips^2 \stsz^2 \agents}{\strconv^2(1 - \wt{\specnorm})} \norm{\avgdmat(t)^{\top} - \dvar^{\ast}}^2 \nn\\
         & \hspace{0.5cm} + \frac{4 \stsz^2}{\strconv^2(1 - \wt{\specnorm})} \norm{\gtmat(t) - \ones \grd \av{\Objfunc} \big(\dmat(t) \big)}^2 + \frac{2 \constsz^2 \Cgen \quanto^2}{1 - \wt{\specnorm}} \norm{\dmat(t) - \auxdv(t)}^2.
     \end{align}
    \subsection*{Step \(3\): Analyzing the gradient tracking error} Let us consider the gradient tracking error \(\gterr(t)\) for each \(t\). First, we consider the term,
    \begin{align*}
        &\norm{\gtmat(t+1) - \ones \avggtmat(t+1)}^2\\
        & = \norm{\gtmat(t+1) - \ones \avggtmat(t) + \ones \avggtmat(t) - \ones \avggtmat(t+1)}^2\\
        & = \norm{\gtmat(t+1) - \ones \avggtmat(t)}^2 + \norm{\ones \avggtmat(t+1) - \ones \avggtmat(t)}^2 \\
        & \hspace{3cm}- 2 
        \sum_{i=1}^{\agents} \inprod{\gtvar_i(t+1) - \avggtmat(t)}{\avggtmat(t+1) - \avggtmat(t)}\\
        & = \norm{\gtmat(t+1) - \ones \avggtmat(t)}^2 + \norm{\ones \avggtmat(t+1) - \ones \avggtmat(t)}^2 \\
        & \hspace{3cm} - 2 \agents \inprod{\avggtmat(t+1) - \avggtmat(t)}{\avggtmat(t+1) - \avggtmat(t)}\\
        & = \norm{\gtmat(t+1) - \ones \avggtmat(t)}^2 - \agents \norm{\ones \avggtmat(t+1) - \ones \avggtmat(t)}^2\\
        & \le \norm{\gtmat(t+1) - \ones \avggtmat(t)}^2.
    \end{align*}
    Second, we now analyze \(\norm{\gtmat(t+1) - \ones \avggtmat(t)}\). Expanding the term \(\gtmat(t+1)\) using \eqref{eq:update rule}, we obtain
    \begin{align*}
        &\norm{\gtmat(t+1) - \ones \avggtmat(t)}^2\\
        & = \norm{\gtmat(t) - \constsz\big(\egtm(t) - \wegtm(t) \big) + \grd \Objfunc\big(\dmat(t+1)\big) - \Objfunc\big(\dmat(t)\big) - \ones \avggtmat(t)}^2 \\
        & = \norm{\gtmat(t) - \constsz\big(\identity{\agents} - \Wght \big)\egtm(t) + \grd \Objfunc\big(\dmat(t+1)\big) - \Objfunc\big(\dmat(t)\big) - \ones \avggtmat(t)}^2\\
        & = \left \lVert\gtmat(t) + \constsz\big(\identity{\agents} - \Wght \big)\gtmat(t) - \constsz\big(\identity{\agents} - \Wght \big)\gtmat(t) - \constsz\big(\identity{\agents} - \Wght \big)\egtm(t) \right.\\ & \hspace{5cm} + \left. \grd \Objfunc\big(\dmat(t+1)\big) - \Objfunc\big(\dmat(t)\big) - \ones \avggtmat(t)\right \rVert^2\\
        & = \left\lVert \big((1 - \constsz )\identity{\agents} + \constsz \Wght\big)\gtmat(t) - \ones \avggtmat(t) + \constsz (\identity{\agents} - \Wght)\big( \gtmat(t) - \egtm(t)\big)\right. \\& \hspace{3cm} + \left. \grd \Objfunc\big(\dmat(t+1)\big) - \Objfunc\big(\dmat(t)\big)\right\rVert^2\\
        & \stackrel{(\star)}{\le} (1 + \tau_4) \wt{\specnorm}^2 \norm{\gtmat(t) - \ones \avggtmat(t)}^2 + 2 \constsz^2 \quanto^2 \bigg(1 + \frac{1}{\tau_4} \bigg) \norm{\gtmat(t) - \egtm(t)}^2 \\& \hspace{2cm} + 2 \bigg(1 + \frac{1}{\tau_4} \bigg) \norm{\grd \Objfunc\big(\dmat(t+1)\big) - \Objfunc\big(\dmat(t)\big)}^2;
    \end{align*}
    here \((\star)\) is obtained by using Lemma \ref{prop:aux_res_1} and subsequently, referring to the discussion preceding Assumption \ref{assum:Standard assumptions}. Recall from Lemma \ref{it:port_6} that
    \begin{align*}
        &\norm{\grd \Objfunc\big(\dmat(t+1)\big) - \Objfunc\big(\dmat(t)\big)}^2 \le 3 \lips^2 \constsz^2  \quanto^2 \norm{\dmat(t) - \edvm(t)}^2 \\& \hspace{1cm}+ 3 \lips^2 \constsz^2  \quanto^2 \norm{\dmat(t) - \ones \avgdmat(t)}^2 + \frac{3\lips^2 \stsz^2}{\strconv^2} \norm{\gtmat(t)}^2,
    \end{align*}
    and therefore, substituting the upper bound for \(\norm{\gtmat(t)}^2\) from \eqref{eq:consen_error_gtmat} in \((\star\star)\) below, we get
    \begin{align*}
        &\norm{\gtmat(t+1) - \ones \avggtmat(t)}^2\\
        & \le (1 + \tau_4) \wt{\specnorm}^2 \norm{\gtmat(t) - \ones \avggtmat(t)}^2 + 2 \constsz^2 \quanto^2 \bigg(1 + \frac{1}{\tau_4} \bigg) \norm{\gtmat(t) - \egtm(t)}^2 \\&  + 2 \bigg(1 + \frac{1}{\tau_4} \bigg) \Bigg( 3 \lips^2 \constsz^2  \quanto^2 \norm{\dmat(t) - \edvm(t)}^2 \\& \hspace{2cm} + 3 \lips^2 \constsz^2  \quanto^2 \norm{\dmat(t) - \ones \avgdmat(t)}^2 + \frac{3\lips^2 \stsz^2}{\strconv^2} \norm{\gtmat(t)}^2\Bigg)\\
        & \stackrel{(\star\star)}{\le} (1 + \tau_4) \wt{\specnorm}^2 \norm{\gtmat(t) - \ones \avggtmat(t)}^2 + 2 \constsz^2 \quanto^2 \bigg(1 + \frac{1}{\tau_4} \bigg) \norm{\gtmat(t) - \egtm(t)}^2 \\&  + 6 \lips^2 \constsz^2  \quanto^2 \bigg(1 + \frac{1}{\tau_4} \bigg)   \norm{\dmat(t) - \edvm(t)}^2 + 6 \lips^2 \constsz^2  \quanto^2 \bigg(1 + \frac{1}{\tau_4} \bigg) \norm{\dmat(t) - \ones \avgdmat(t)}^2\\ & \hspace{0.5cm} + \frac{6\lips^2 \stsz^2}{\strconv^2} \bigg(1 + \frac{1}{\tau_4} \bigg) \Bigg( (1+ \tau_3) \norm{\gtmat(t) - \ones \grd \av{\Objfunc} \big(\dmat(t) \big)}^2 + 2 \lips^2 \bigg(1 + \frac{1}{\tau_3} \bigg) \norm{\dmat(t) - \ones \avgdmat(t)}^2 \nn \\ & \hspace{2.5cm} + 2 \agents \lips^2 \bigg(1 + \frac{1}{\tau_3} \bigg) \norm{\avgdmat(t)^{\top} - \dvar^{\ast}}^2\Bigg).
    \end{align*}
    Recall from Lemma \ref{it:error_comp_2} that
    \begin{equation*}
        \begin{cases}
            \EE \cexpecof[\bigg]{\norm{\dmat(t) - \edvm(t)}^2 \given \sigalg_t} \le \Cgen \norm{\dmat(t) - \auxdv(t)}^2 \text{ }\\
            \EE \cexpecof[\bigg]{\norm{\gtmat(t) - \egtm(t)}^2 \given  \sigalg_t} \le \Cgen \norm{\gtmat(t) - \auxgt(t)}^2.
        \end{cases}
    \end{equation*}
    Stacking similar terms together, and in view of the above inequalities, the resultant expression is given by
    \begin{align*}
        &\EE \cexpecof[\Big]{\norm{\gtmat(t+1) - \ones \avggtmat(t)}^2 \given \sigalg_t}\\
        & \le \frac{12 \lips^4 \stsz^2 \agents}{\strconv^2} \bigg(1 + \frac{1}{\tau_4} \bigg)\bigg(1 + \frac{1}{\tau_3} \bigg) \norm{\avgdmat(t)^{\top} - \dvar^{\ast}}^2 \\
        & + \Bigg(  6 \lips^2 \constsz^2  \quanto^2 \bigg(1 + \frac{1}{\tau_4} \bigg) + \frac{12\lips^4 \stsz^2}{\strconv^2} \bigg(1 + \frac{1}{\tau_4} \bigg) \bigg(1 + \frac{1}{\tau_3} \bigg) \Bigg) \norm{\dmat(t) - \ones \avgdmat(t)}^2\\
        & + \Bigg((1 + \tau_4) \wt{\specnorm}^2 + \frac{6\lips^2 \stsz^2}{\strconv^2} \bigg(1 + \frac{1}{\tau_4} \bigg)  (1+ \tau_3) \Bigg) \norm{\gtmat(t) - \ones \grd \av{\Objfunc} \big(\dmat(t) \big)}^2 \\
        & + 6 \lips^2 \constsz^2  \quanto^2 \Cgen \bigg(1 + \frac{1}{\tau_4} \bigg)   \norm{\dmat(t) - \auxdv(t)}^2  + 2 \constsz^2 \quanto^2 \Cgen \bigg(1 + \frac{1}{\tau_4} \bigg) \norm{\gtmat(t) - \auxgt(t)}^2. 
    \end{align*}
    Choose \(\tau_3 = 1\) and \(\tau_4 = \frac{1 - \wt{\specnorm}^2}{2 \wt{\specnorm}^2}\), we simplify the coefficients in a similar way. The details are given below:
    \begin{itemize}[leftmargin=*, label = \(\circ\)]
        \item Note that \(1 + \tau_4 = \frac{1+\wt{\specnorm}^2}{2 \wt{\specnorm}}\) and \(1 + \frac{1}{\tau_4} = \frac{1+\wt{\specnorm}^2}{1-\wt{\specnorm}^2} \le \frac{1}{1 - \wt{\specnorm}} \frac{1+\wt{\specnorm}^2}{1+\wt{\specnorm}} \le \frac{1}{1 - \wt{\specnorm}}\).

        \item Let us compute the coefficients:
        \begin{align*}
            \begin{cases}
             \frac{12 \lips^4 \stsz^2 \agents}{\strconv^2} \bigg(1 + \frac{1}{\tau_4} \bigg)\bigg(1 + \frac{1}{\tau_3} \bigg) =  \frac{24 \lips^4 \stsz^2 \agents}{\strconv^2} \frac{1+\wt{\specnorm}^2}{1-\wt{\specnorm}^2} \le 
             \frac{24 \lips^4 \stsz^2 \agents}{\strconv^2(1 - \wt{\specnorm})} \vspace{1.5mm}\\
             6 \lips^2 \constsz^2  \quanto^2 \bigg(1 + \frac{1}{\tau_4} \bigg) + \frac{12\lips^4 \stsz^2}{\strconv^2} \bigg(1 + \frac{1}{\tau_4} \bigg) \bigg(1 + \frac{1}{\tau_3} \bigg) \le 
             \frac{6 \lips^2 \constsz^2 \quanto^2}{1 - \wt{\specnorm}} + \frac{24 \lips^4 \stsz^2}{\strconv^2(1 - \wt{\specnorm})} \vspace{1.5mm}\\
             (1 + \tau_4) \wt{\specnorm}^2 + \frac{6\lips^2 \stsz^2}{\strconv^2} \bigg(1 + \frac{1}{\tau_4} \bigg)  (1+ \tau_3) \le
             \frac{1+ \wt{\specnorm}^2}{2} + \frac{12 \lips^2 \stsz^2}{\strconv^2(1 - \wt{\specnorm})} 
            \vspace{1.5mm}\\
             6 \lips^2 \constsz^2  \quanto^2 \Cgen \bigg(1 + \frac{1}{\tau_4} \bigg) \le \frac{6 \lips^2 \constsz^2 \quanto^2 \Cgen}{1 - \wt{\specnorm}} \vspace{1.5mm}\\
             2 \constsz^2 \quanto^2 \Cgen \bigg(1 + \frac{1}{\tau_4} \bigg) \le \frac{2 \constsz^2 \quanto^2 \Cgen}{1 - \wt{\specnorm}}.
            \end{cases}
        \end{align*}
    \end{itemize}
    The final expression for the gradient tracking error is given by 
    \begin{align}
        \label{eq:final_GT_error}
        &\EE \expecof[\Big]{\norm{\gtmat(t+1) - \ones \avggtmat(t+1)}^2}\nn\\
        & \le \frac{24 \lips^4 \stsz^2 \agents}{\strconv^2(1 - \wt{\specnorm})} \norm{\avgdmat(t)^{\top} - \dvar^{\ast}}^2 + \bigg( \frac{6 \lips^2 \constsz^2 \quanto^2}{1 - \wt{\specnorm}} + \frac{24 \lips^4 \stsz^2}{\strconv^2(1 - \wt{\specnorm})} \bigg) \norm{\dmat(t) - \ones \avgdmat(t)}^2 \nn \\
        & \hspace{2cm} + \bigg(\frac{1+ \wt{\specnorm}^2}{2} + \frac{12 \lips^2 \stsz^2}{\strconv^2(1 - \wt{\specnorm})} \bigg) \norm{\gtmat(t) - \ones \grd \av{\Objfunc} \big(\dmat(t) \big)}^2 \nn \\
        & \hspace{1cm}+ \frac{6 \lips^2 \constsz^2 \quanto^2 \Cgen}{1 - \wt{\specnorm}} \norm{\dmat(t) - \auxdv(t)}^2 + \frac{2 \constsz^2 \quanto^2 \Cgen}{1 - \wt{\specnorm}} \norm{\gtmat(t) - \auxgt(t)}^2.
    \end{align}
    \subsection*{Step \(4\): Analyzing the compression errors associated with the decision variable} We now focus on the compression error. In Algorithm \ref{alg:sec_ord_comp}, the update rule to update \(\auxdv(t)\) for each \(t\) is given by \(\auxdv(t+1) = (1 - \condv)\auxdv(t) + \condv \edvm(t)\). Moreover, we recall from Definition \ref{def:General Compression Operator} the general compressor operator. In Algorithm \ref{alg:sec_ord_comp}, we adhere to \(\compscheme \bigl(\dmat(t) - \auxdv(t) \bigr) \equiv \gcompop \big( \dmat(t) - \auxdv(t)\big)\).
    With these notations, we now modify the compression error as follows: for each iteration \(t\),
    \begin{align*}
        &\norm{\dmat(t+1) - \auxdv(t+1)}^2\\
        & = \norm{\dmat(t+1) - (1 - \condv)\auxdv(t) - \condv \edvm(t)}^2\\
        & = \norm{\dmat(t+1) - (1 - \condv)\auxdv(t) - \condv \gcompop \big( \dmat(t) - \auxdv(t) \big) -  \condv \auxdv(t)}^2 \\
        & = \norm{\dmat(t+1) - \auxdv(t) - \condv \gcompop \big( \dmat(t) - \auxdv(t) \big)}^2.
    \end{align*}
    We use the shorthand notations
    \begin{equation*}
        \encode{\dvar}(t) \Let \gcompop \big( \dmat(t) - \auxdv(t) \big) \text{ and } \encode{\dvar}^{\scale}(t) \Let \frac{\gcompop \big( \dmat(t) - \auxdv(t) \big)}{\scale}. 
    \end{equation*}
    Therefore, for \(\scale>0\) defined as in Definition \ref{def:General Compression Operator},
    \begin{align*}
        &\norm{\dmat(t+1) - \auxdv(t) - \condv \encode{\dvar}(t)}^2\\
        & = \left \lVert\dmat(t+1) - \dmat(t) + \dmat(t) - \auxdv(t) - \condv \encode{\dvar}(t) \right.\\& \hspace{3cm} \left. + \condv \scale \big(\dmat(t) - \auxdv(t) \big) - \condv \scale \big(\dmat(t) - \auxdv(t) \big)\right\rVert^2 \\
        & = \left \lVert \dmat(t+1) - \dmat(t) + (1 - \condv \scale) \big(\dmat(t) - \auxdv(t)\big)  \right.\\& \hspace{3cm} \left. + \condv \scale \big(\dmat(t) - \auxdv(t)  - \encode{\dvar}^{\scale}(t)\big)  \right\rVert^2 \\
        & \le \taudv \norm{(1 - \condv \scale) \big(\dmat(t) - \auxdv(t)\big)  + \condv \scale \big(\dmat(t) - \auxdv(t)  - \encode{\dvar}^{\scale}(t)\big)}^2 \\& \hspace{3cm} + \frac{\taudv}{\taudv - 1}\norm{\dmat(t+1) - \dmat(t)}^2 \\ 
        & \le \taudv \bigg(\condv \scale \norm{\dmat(t) - \auxdv(t)  - \encode{\dvar}^{\scale}(t)}^2 + (1 - \condv \scale) \norm{ \big(\dmat(t) - \auxdv(t)\big)}^2 \bigg)\\
        & \hspace{2cm} + \frac{\taudv}{\taudv - 1}\norm{\dmat(t+1) - \dmat(t)}^2 \quad \text{(from convexity of norm)}
    \end{align*}
    for any \(\taudv>1\) and \(1 - \condv \scale \ge 0\). This is implies that \(0 < \condv \le \frac{1}{\scale}\). Obviously, if \(\scale=1\), then \(\condv \in \lorc{0}{1}\). Taking the conditional expectation on both sides, we get
    \begin{align}\label{eq:pre_comp_error_}
        &\EE \cexpecof[\big]{\norm{\dmat(t+1) - \auxdv(t+1)}^2 \given \sigalg_t} \nn\\
        & \le \taudv \condv \scale \EE \cexpecof[\big]{\norm{\dmat(t) - \auxdv(t)  - \encode{\dvar}^{\scale}(t)}^2 \given \sigalg_t} + \taudv (1 - \condv \scale)\norm{ \big(\dmat(t) - \auxdv(t)\big)}^2 \nn \\
        & \hspace{2cm} + \frac{\taudv}{\taudv - 1}\EE \cexpecof[\big]{\norm{\dmat(t+1) - \dmat(t)}^2 \given \sigalg_t} \nn \\ 
        & \le \taudv \condv \scale (1 - \compar) \norm{\dmat(t) - \auxdv(t) }^2 + \taudv (1 - \condv \scale)\norm{ \big(\dmat(t) - \auxdv(t)\big)}^2 \nn \\
        & \hspace{2cm} + \frac{\taudv}{\taudv - 1}\EE \cexpecof[\big]{\norm{\dmat(t+1) - \dmat(t)}^2 \given \sigalg_t} \nn \\ 
        & = \taudv \Big(\underbrace{\condv \scale (1 - \compar) +  (1 - \condv \scale)}_{= 1 - \condv \scale \compar} \Big) \norm{ \big(\dmat(t) - \auxdv(t)\big)}^2 \nn \\& \hspace{3cm} + \frac{\taudv}{\taudv - 1}\EE \cexpecof[\big]{\norm{\dmat(t+1) - \dmat(t)}^2 \given \sigalg_t}
    \end{align}
    for some constant \(\compar \in \lorc{0}{1}\). Substituting the expression for \(\norm{\dmat(t+1) - \dmat(t)}^2\) from \eqref{eq:crucial_step}, we have
    \begin{align*}
        &\frac{\taudv}{\taudv-1}\EE \cexpecof[\big]{\norm{\dmat(t+1) - \dmat(t)}^2\given \sigalg_t}  \le 3 \frac{\taudv}{\taudv-1} \constsz^2 \quanto^2 \EE \cexpecof[\Big]{\norm{\dmat(t) - \edvm(t)}^2\given \sigalg_t} \\& \hspace{2cm}+ 3 \frac{\taudv}{\taudv-1} \constsz^2 \quanto^2 \norm{\dmat(t) - \ones \avgdmat(t)}^2 + \frac{\taudv}{\taudv-1} \frac{3 \stsz^2}{\strconv^2} \norm{\gtmat(t)}^2. 
    \end{align*}
    Substituting back \(\norm{\gtmat(t)}^2\) from \eqref{eq:consen_error_gtmat} with \(\tau_3 = 1\) to obtain,
    \begin{align*}
        &\frac{\taudv}{\taudv-1}\EE \cexpecof[\big]{\norm{\dmat(t+1) - \dmat(t)}^2\given \sigalg_t} \le 3 \frac{\taudv}{\taudv-1} \constsz^2 \quanto^2 \Cgen \norm{\dmat(t) - \auxdv(t)}^2  \\& \hspace{0.5cm}+ 3 \frac{\taudv}{\taudv-1} \constsz^2 \quanto^2 \norm{\dmat(t) - \ones \avgdmat(t)}^2 + \frac{\taudv}{\taudv-1} \frac{3 \stsz^2}{\strconv^2} \bigg( (1+ \tau_3) \norm{\gtmat(t) - \ones \grd \av{\Objfunc} \big(\dmat(t) \big)}^2 \nn \\ & \hspace{0.5cm} + 2 \lips^2 \bigg(1 + \frac{1}{\tau_3} \bigg) \norm{\dmat(t) - \ones \avgdmat(t)}^2  + 2 \agents \lips^2 \bigg(1 + \frac{1}{\tau_3} \bigg) \norm{\avgdmat(t)^{\top} - \dvar^{\ast}}^2\bigg)\\
        & =  \frac{12 \lips^2 \stsz^2 \agents}{\strconv^2} \frac{\taudv}{\taudv - 1} \norm{\avgdmat(t)^{\top} - \dvar^{\ast} }^2 + \bigg(3 \constsz^2 \quanto^2 + \frac{12 \lips^2 \stsz^2}{\strconv^2} \bigg) \frac{\taudv}{\taudv - 1} \norm{\dmat(t) - \ones \avgdmat(t)}^2 \\
        & \hspace{0.5cm}+ \frac{6\stsz^2}{\strconv^2}\frac{\taudv}{\taudv - 1} \norm{\gtmat(t) - \ones \avggtmat(t)}^2 + 3 \constsz^2 \quanto^2 \Cgen \frac{\taudv}{\taudv - 1} \norm{\dmat(t) - \auxdv(t)}^2,
    \end{align*}
    which we substitute back to \eqref{eq:pre_comp_error_} to get
    \begin{align}
        \label{eq:final_comp_error}
        &\EE \cexpecof[\big]{\norm{\dmat(t+1) - \auxdv(t+1)}^2 \given \sigalg_t} \nn \\
        & \le \taudv \big(1 - \condv \scale \compar \big) \norm{ \big(\dmat(t) - \auxdv(t)\big)}^2 + \frac{\taudv}{\taudv - 1}\EE \cexpecof[\big]{\norm{\dmat(t+1) - \dmat(t)}^2 \given \sigalg_t} \nn \\
        & \le \frac{12 \lips^2 \stsz^2 \agents}{\strconv^2} \frac{\taudv}{\taudv - 1} \norm{\avgdmat(t)^{\top} - \dvar^{\ast} }^2 + \bigg(3 \constsz^2 \quanto^2 + \frac{12 \lips^2 \stsz^2}{\strconv^2} \bigg) \frac{\taudv}{\taudv - 1} \norm{\dmat(t) - \ones \avgdmat(t)}^2 \nn\\
        & \hspace{2.5cm}+ \frac{6\stsz^2}{\strconv^2}\frac{\taudv}{\taudv - 1} \norm{\gtmat(t) - \ones \avggtmat(t)}^2 \nn\\& \hspace{2cm} + \bigg(3 \constsz^2 \quanto^2 \Cgen \frac{\taudv}{\taudv - 1} +  \taudv \big(1 - \condv \scale \compar \big)\bigg)\norm{\dmat(t) - \auxdv(t)}^2\nn \\
        &= \frac{4 \lips^2 \stsz^2 \agents \tdv}{\strconv^2} \norm{\avgdmat(t)^{\top} - \dvar^{\ast} }^2 + \bigg(\constsz^2 \quant{1} + \frac{4\lips^2 \stsz^2 \tdv} {\strconv^2}\bigg) \norm{\dmat(t) - \ones \avgdmat(t)}^2 \nn \\
        & \hspace{1cm}+ \frac{2\stsz^2 \tdv}{\strconv^2} \norm{\gtmat(t) - \ones \avggtmat(t)}^2 + \big(\quantdv + \constsz^2 \quant{2}\big) \norm{\dmat(t) - \auxdv(t)}^2,
    \end{align}
    where for \(\scale>0\) and \(\compar \in \lorc{0}{1}\), we have \(\taudv(1 -  \condv \scale \compar) < \taudv\) and \(\taudv>1\) from Proposition \ref{prop:aux_res_1}. Choose \(\taudv>1\) such that \(\taudv(1 - \condv \scale \compar) < 1\), implying that \(1 < \taudv < \frac{1}{1 - \condv \scale\compar}\). Of course, if \(\compar = 1\) and \(\condv  = \frac{1}{\scale}\), then \(\taudv(1 -  \condv \scale \compar) = 0\) and we can choose any arbitrary \(\taudv>1\).

    \subsection*{Step \(5\): Analyzing the compression errors associated with the gradient tracker} Finally, we analyze \(\comerrgt(t)\). The steps that are similar to Step \(4\) are omitted. We begin by noticing (similar to Step \(4\)) that 
    \begin{align*}
        &\EE \cexpecof[\big]{\norm{\gtmat(t+1) - \auxgt(t+1)}^2 \given \sigalg_t}\\
        & \le \taugt \big( 1 - \congt \scale \compar\big)\norm{\gtmat(t) - \auxgt(t)}^2 + \frac{\taugt}{\taugt-1}\EE \cexpecof[\big]{\norm{\gtmat(t+1) - \gtmat(t)}^2 \given \sigalg_t},
    \end{align*}
    where \(\scale>0\), \(\compar \in \lorc{0}{1}\), and for any \(\taugt>1\) with \(1 - \congt \scale \ge 0\). This implies that \(\congt \in \lorc{0}{\frac{1}{\scale}}\). Let us now focus on \(\cexpecof[\big]{\norm{\gtmat(t+1) - \gtmat(t)}^2 \given \sigalg_t}\). We observe that
    \begin{align*}
        &\norm{\gtmat(t+1) - \gtmat(t)}^2\\
        & = \norm{\gtmat(t+1) + \constsz \big(\ones \avggtmat(t) - \ones \avggtmat(t)\big) - \gtmat(t)}^2\\
        & = \left\lVert\constsz(\identity{\agents} - \Wght)\gtmat(t) - \constsz(\identity{\agents} - \Wght)\gtmat(t) - \constsz(\identity{\agents} - \Wght)\egtm(t) \right. +\underbrace{\constsz \big(\ones \avggtmat(t) - \ones \avggtmat(t)\big)}_{= \constsz(\identity{\agents} - \Wght)\ones \avggtmat(t)} \\
        & \hspace{3cm} \left.+ \grd \Objfunc\big(\dmat(t+1)\big) - \grd \Objfunc\big(\dmat(t)\big)\right\rVert^2 \\
        & = \left \lVert \constsz(\identity{\agents} - \Wght)\big(\gtmat(t) - \egtm(t)\big) -  \constsz(\identity{\agents} - \Wght)\big(\gtmat(t) - \ones \avggtmat(t) \big) \right. \\ 
        & \hspace{3cm} \left. + \Objfunc\big(\dmat(t+1)\big) - \grd \Objfunc\big(\dmat(t)\big) \right \rVert^2 \\
        & \le 3 \constsz^2 \quanto^2 \norm{\gtmat(t) - \egtm(t)}^2 + 3 \constsz^2 \quanto^2 \norm{\gtmat(t) - \ones \avggtmat(t)}^2 + 3 \norm{\Objfunc\big(\dmat(t+1)\big) - \grd \Objfunc\big(\dmat(t)\big)}^2,
    \end{align*}
    implying that for each iterations \(t\), we obtain
    \begin{align*}
       \EE \cexpecof[\big]{\norm{\gtmat(t+1) - \gtmat(t)}^2 \given\sigalg_t } &\le 3 \constsz^2 \quanto^2 \Cgen \norm{\gtmat(t) - \auxgt(t)}^2  + 3 \constsz^2 \quanto^2 \norm{\gtmat(t) - \ones \avggtmat(t)}^2 \\
       & \hspace{3cm} + 3 \EE \cexpecof[\Big]{\norm{\Objfunc\big(\dmat(t+1)\big) - \grd \Objfunc\big(\dmat(t)\big)}^2 \given 
       \sigalg_t}.
    \end{align*}
    Put the upper bound of \(\norm{\gtmat(t)}^2\) from \eqref{eq:consen_error_gtmat} in Lemma \ref{it:port_6}, choose \(\tau_3 = 1\) as before, we recover 
    \begin{align}
        \label{eq:com_err_gt_pre_final}
        & 3 \EE \cexpecof[\Big]{\norm{\Objfunc\big(\dmat(t+1)\big) - \grd \Objfunc\big(\dmat(t)\big)}^2 \given 
       \sigalg_t} \nn\\
       & \le 3 \Bigg( \frac{12 \lips^4 \stsz^2 \agents}{\strconv^2} \norm{\avgdmat(t) - \dvar^{\ast}}^2 + \bigg(3 \lips^2 \constsz^2 \quanto^2 + \frac{12 \lips^4 \stsz^2}{\strconv^2} \bigg)\norm{\dmat(t) - \ones \avgdmat(t)}^2 \nn \\
       & \hspace{1cm} + \frac{6 \lips^2 \stsz^2}{\strconv^2}\norm{\gtmat(t) - \ones \avggtmat(t)}^2 + 3 \lips^2 \constsz^2 \quanto^2 \Cgen \norm{\dmat(t) - \auxdv(t)}^2 \Bigg).
    \end{align}
    Therefore, 
    \begin{align*}
        & \EE \cexpecof[\big]{\norm{\gtmat(t+1) - \gtmat(t)}^2 \given\sigalg_t }\\
        & \le \frac{36 \lips^4 \stsz^2 \agents}{\strconv^2} \norm{\avgdmat(t) - \dvar^{\ast}}^2 + \bigg(9 \lips^2 \constsz^2 \quanto^2 + \frac{36 \lips^4 \stsz^2}{\strconv^2} \bigg)\norm{\dmat(t) - \ones \avgdmat(t)}^2 \nn \\
       & \hspace{1cm} + \bigg(3\constsz^2\quanto^2 + \frac{18 \lips^2 \stsz^2}{\strconv^2}\bigg)\norm{\gtmat(t) - \ones \avggtmat(t)}^2 + 9 \lips^2 \constsz^2 \quanto^2 \Cgen \norm{\dmat(t) - \auxdv(t)}^2 \\
       & \hspace{3cm} + 3 \constsz^2 \quanto^2 \Cgen \norm{\gtmat(t) - \auxgt(t)}^2,
    \end{align*}
    which implies that 
    \begin{align}
        \label{eq:com_err_final_gt}
        &\EE \cexpecof[\big]{\norm{\gtmat(t+1) - \auxgt(t+1)}^2 \given \sigalg_t}\nn\\
        & \le \frac{\taugt}{\taugt-1} \Bigg( \frac{36 \lips^4 \stsz^2 \agents}{\strconv^2} \norm{\avgdmat(t) - \dvar^{\ast}}^2 + \bigg(9 \lips^2 \constsz^2 \quanto^2 + \frac{36 \lips^4 \stsz^2}{\strconv^2} \bigg)\norm{\dmat(t) - \ones \avgdmat(t)}^2 \nn \\
       & \hspace{1cm} + \bigg(3\constsz^2\quanto^2 + \frac{18 \lips^2 \stsz^2}{\strconv^2}\bigg)\norm{\gtmat(t) - \ones \avggtmat(t)}^2 + 9 \lips^2 \constsz^2 \quanto^2 \Cgen \norm{\dmat(t) - \auxdv(t)}^2 \nn \\
       & \hspace{3cm} + 3 \constsz^2 \quanto^2 \Cgen \norm{\gtmat(t) - \auxgt(t)}^2 \Bigg) + \taugt \big( 1 - \congt \scale \compar\big)\norm{\gtmat(t) - \auxgt(t)}^2 \nn\\
       & =  \frac{\taugt}{\taugt-1} \Bigg( \frac{36 \lips^4 \stsz^2 \agents}{\strconv^2} \norm{\avgdmat(t) - \dvar^{\ast}}^2 + \bigg(9 \lips^2 \constsz^2 \quanto^2 + \frac{36 \lips^4 \stsz^2}{\strconv^2} \bigg)\norm{\dmat(t) - \ones \avgdmat(t)}^2 \nn \\
       & \hspace{1cm} + \bigg(3\constsz^2\quanto^2 + \frac{18 \lips^2 \stsz^2}{\strconv^2}\bigg)\norm{\gtmat(t) - \ones \avggtmat(t)}^2 + 9 \lips^2 \constsz^2 \quanto^2 \Cgen \norm{\dmat(t) - \auxdv(t)}^2\Bigg) \nn\\
       & \hspace{2cm} + \bigg( 3 \constsz^2 \quanto^2 \Cgen \frac{\taugt}{\taugt-1} + \taugt \big( 1 - \congt \scale \compar\big) \bigg) \norm{\gtmat(t) - \auxgt(t)}^2,
       \end{align}
       where \(\scale>0\) and \(\compar \in \lorc{0}{1}\), we have \(\taugt(1 -  \congt \scale \compar) < \taugt\) and \(\taugt>1\) from Proposition \ref{prop:aux_res_1}. With similar arguments as in Step \(4\), we choose \(\taugt > 1\) such that \(\taugt(1 - \congt \scale \compar) < 1\), implying that \(1 < \taugt < \frac{1}{1 - \congt \scale\compar}\).
    Recall the notations established in \eqref{eq:notations}. Then \eqref{eq:com_err_final_gt} reduces to
    \begin{align}
        \label{eq:compression error grad track}
        &\EE \cexpecof[\big]{\norm{\gtmat(t+1) - \auxgt(t+1)}^2 \given \sigalg_t}\nn\\
        & = \frac{12 \lips^4 \stsz^2 \agents \tgt}{\strconv^2} \norm{\avgdmat(t) - \dvar^{\ast}}^2 + \bigg( 3 \lips^2 \constsz^2 \quant{3} + \frac{12 \lips^4 \stsz^2 \tgt}{\strconv^2} \bigg) \norm{\dmat(t) - \ones \avgdmat(t)}^2 \nn \\
       & \hspace{0.5cm} + \bigg(\constsz^2 \quant{3} + \frac{6 \lips^2 \stsz^2 \tgt}{\strconv^2} \bigg)\norm{\gtmat(t) - \ones \avggtmat(t)}^2 + 3 \lips^2 \constsz^2 \quant{4}\norm{\dmat(t) - \auxdv(t)}^2 \nn \\
       & \hspace{3cm} + (\quantgt + \constsz^2 \quant{4})\norm{\gtmat(t) - \auxgt(t)}^2 
    \end{align}
    The final expressions of the upper bounds of the errors derived --- \eqref{eq:opt err}, \eqref{eq:final_exp_con_error}, \eqref{eq:pre_comp_error_}, and \eqref{eq:final_GT_error}, \eqref{eq:compression error grad track} ---  are now stacked together to get \(\err(t+1) \le \conmat(\hypar) \err(t)\) for each iterations \(t\), whenever \(\stsz \le \frac{2\lips}{3 \strconv}\). This completes the first part of the proof of Lemma \ref{lem: key lemma}.
\end{proof}

\subsection{Proof of Theorem \ref{th:Main result_str_conv_case}}\label{subappen:main proof strong convex}
 We begin by establishing a set of sufficient conditions such that 
      \begin{align*}
          \conmat\big(\hypar\big) \eps \le \bigg( 1 - \frac{\stsz}{2 \cn}\bigg) \eps
      \end{align*}
      where \(\eps\) is a positive vector. Consider the first row
      \begin{align}\label{eq:fgh}
          \bigg(1 - \frac{3 \stsz}{2 \cn} + \frac{\stsz^3}{2 \cn^3}\bigg) \eps_1 + \bigg(\frac{\stsz^2 \cn^2}{\agents} + \frac{2 \stsz \cn^3}{\agents} \bigg) \eps_2 + \bigg(\frac{\stsz^2 }{\strconv^2 \agents} + \frac{2 \stsz^2 \lips}{\stsz \strconv^3 \agents} \bigg) \eps_3 \le \bigg(1 - \frac{\stsz}{2 \cn} \bigg) \eps_1,
      \end{align}
      which is equivalent to
      \(
          \frac{\stsz^3}{2 \cn^2} \eps_1 + \big(\frac{\stsz^2 \cn^2}{\agents} +  \frac{2 \stsz \cn^3}{\agents} \big)\eps_2 + \big(\frac{\stsz^2 }{\strconv^2 \agents} + \frac{2 \stsz \cn}{\strconv^2 \agents} \big) \eps_3 \le \frac{\stsz }{\cn}\eps_1
      \). If \(\stsz\) is chosen such that 
      \begin{align}\label{eq:suff_cond_1}
          \stsz \le \min \aset[\Bigg]{\frac{\eps_1}{\frac{3 \cn^3 \eps_2}{\agents} + \frac{3 \cn\eps_3}{\strconv^2 \agents}}, \cn \sqrt{\frac{2}{3}}} \text{ and }\frac{\eps_1}{\eps_2} \ge \frac{6 \cn^4}{\agents} + \frac{6 \cn^2}{\strconv^2 \agents} \frac{\eps_3}{\eps_2},
      \end{align}
    then the inequality \eqref{eq:fgh} is satisfied.
    Let us consider the second row, which is given by
    \begin{align}\label{eq:suff_con_aux}
        \frac{8 \cn^2 \stsz^2 \agents}{\strconv^2 \wh{\specnorm}} \eps_1 + \Bigg(\frac{1 + \wt{\specnorm}^2}{2} + \frac{8 \cn^2 \stsz^2}{\wh{\specnorm}} \Bigg) \eps_2 + \frac{4 \cn^2 \stsz^2 }{\wt{\specnorm}}\eps_3 + \frac{2 \constsz^2 \quanto^2 \Cgen^2}{\wh{\specnorm}}\eps_4 \le \bigg(1 - \frac{\stsz}{2 \cn} \bigg) \eps_2,
    \end{align}
    which can be simplified as
    \begin{align*}
        \frac{8 \cn^2 \stsz^2 \agents}{\strconv^2 \wh{\specnorm}} \eps_1 +  \frac{8 \cn^2 \stsz^2}{\wh{\specnorm}}  \eps_2 + \frac{4 \cn^2 \stsz^2 }{\wt{\specnorm}}\eps_3 + \frac{2 \constsz^2 \quanto^2 \Cgen^2}{\wh{\specnorm}}\eps_4 + \frac{\stsz}{2 \cn} \eps_2 &\le \frac{1 - \wt{\specnorm}^2}{2} \eps_2.
    \end{align*}
    Note that \(\frac{1 - \wt{\specnorm}^2}{2} = \frac{\constsz(1 - \specnorm)(2 - \constsz(1 - \specnorm))}{2} \ge \frac{\constsz(1- \specnorm)}{4}\), if \(\specnorm < 1\) and \(\constsz \le \frac{1.5}{1 - \specnorm}\) so that \((2 - \constsz(1 - \specnorm)) \ge 1\). In view of this, \eqref{eq:suff_con_aux} holds if
    \begin{align}\label{eq:step_}
        \frac{4 \cn^2 \stsz^2}{\wh{\specnorm}}\big( 2 \agents \eps_1+ 2 \eps_2 + \eps_3\big) + \frac{\stsz }{2 \cn} \eps_2  + \frac{2 \constsz^2 \quanto^2 \Cgen^2}{\wh{\specnorm}}\eps_4 \le \frac{\constsz(1 - \wt{\specnorm})}{4} \eps_2.
    \end{align}
    If we pick \(\eps\) and \(\stsz \) such that
    \begin{align}\label{eq:suff_row2}
        &\stsz \le \min \aset[\Bigg]{\frac{\constsz (1 - \wt{\specnorm}) \cn}{6}, \frac{1}{4 \cn} \sqrt{\frac{\constsz(1 - \specnorm)\wh{\specnorm}\eps_2}{3\wh{\epsln}}}} \text{ and }
        \frac{\eps_4}{\eps_2} \le \frac{(1 - \specnorm)\wh{\specnorm}}{8 \constsz^2 \quanto^2 \Cgen^2},
    \end{align}
      \eqref{eq:step_} holds, which in turn implies that \eqref{eq:suff_con_aux} remains valid. Now let us consider the third row
      \begin{align}\label{eq:gt_step}
          &\frac{24 \cn^2 \stsz^2 \agents}{\wh{\specnorm}} \eps_1 + \Bigg( \frac{6 \constsz^2 \quanto^2 }{\wh{\specnorm}} + \frac{24 \cn^2 \stsz^2}{\wh{\specnorm}}\Bigg) \eps_2 + \Bigg( \frac{1+ \wt{\specnorm^2}}{2} + \frac{12 \cn^2 \stsz^2}{\wh{\specnorm}}\Bigg) \lips^2 \eps_3 +  \frac{6 \constsz^2 \quanto^2 \Cgen}{\wh{\specnorm}}\eps_4 \nn \\
          & \hspace{1cm} + \frac{2 \constsz^2 \quanto^2 \Cgen}{\wh{\specnorm}}\eps_5 \le \Bigg(1  - \frac{\stsz}{2 \cn}\Bigg) \eps_3.
      \end{align}
      Using similar arguments, we can show that if
      \begin{align}
          \label{eq:gt_final}
          &\stsz \le \min \aset[\Bigg]{\frac{\constsz(1 - \specnorm)\cn}{6}, \frac{1}{12 \cn}\sqrt{\frac{\constsz(1 - \specnorm)\wh{\specnorm}\eps_3}{\wh{\epsln}}}} \nn \\
          &3 \eps_2 + 3 \Cgen \eps_4 + \Cgen \eps_5  \le \frac{(1 - \specnorm)\wt{\specnorm}\eps_3}{24 \constsz \quanto^2},
      \end{align}
      then \eqref{eq:gt_step} remains valid. We now consider the fourth row, given by
      \begin{align}\label{eq:tyb}
          4 \cn^2 \stsz^2 \agents \tdv \eps_1 + \big(\constsz^2 \quant{1} + 4 \cn2 \stsz^2 \tdv \big) \eps_2 + 2 \cn^2 \stsz^2 \tdv \eps_3 + \big(\quantdv + \constsz^2 \quant{2}\big) \eps_4 \le \bigg(1 -  \frac{\stsz}{2 \cn}\bigg)\eps_4,
      \end{align}
      which can be simplified as
      \begin{align}\label{eq:ghjk}
          2 \cn^2 \stsz^2 \tdv \big(2 \agents \eps_1 + 2 \eps_2 + \eps_3\big) + \constsz^2 \big(\quant{1} \eps_2 + \quant{2}\eps_4 \big) + \frac{\stsz}{2\cn} \eps_4  \le (1 - \quantdv)\eps_4.
      \end{align}
      If we choose \(\stsz \le \frac{\constsz}{\cn}\), then \eqref{eq:ghjk} can be written as
      \begin{align*}
          \constsz^2 \Big(2 \tdv \big(2 \agents \eps_1 + 2 \eps_2 + \eps_3 \big) + \big(\quant{1} \eps_2 + \quant{2}\eps_4 \big)  \Big) + \frac{\constsz}{2\cn^2} \eps_4 \le (1 - \quantdv)\eps_4,
      \end{align*}
       and upon picking 
      \begin{align}
          \label{eq:final_iuo}
              \constsz \le \min \aset[\Bigg]{1, \sqrt{\frac{(1 - \quantdv)\eps_4}{2 \tdv \wh{\epsln} + \ol{\epsln} + \frac{\epsln_4}{2 \cn^2}}}},
      \end{align}
      then \eqref{eq:ghjk} holds, implying that \eqref{eq:tyb} is valid. A similar line of reasoning can be adopted to get a sufficient condition such that the fifth row remains valid. We omit the details and directly write the sufficient conditions
      \begin{align}
          \label{eq:final_uipu}
              &\stsz \le \frac{\constsz}{\cn} \text{ and } \constsz \le \min \aset[\Bigg]{1, \sqrt{\frac{(1 - \quantgt)\eps_5}{6 \tdv \wh{\epsln} +  \Breve{\epsln} + \frac{\epsln_5}{2 \cn^2}}} }.
      \end{align}
      Combining \eqref{eq:suff_cond_1}, \eqref{eq:suff_row2}, \eqref{eq:gt_final}, \eqref{eq:final_iuo}, and \eqref{eq:final_uipu}, we obtain the sufficient conditions on \((\stsz, \constsz)\), specified in \eqref{eq:stsz_constsz} such that the inequality 
      \begin{align}\label{eq:obj_2}
          \conmat \big( \hypar\big) \eps \le \bigg(1 - \frac{\stsz}{2 \cn} \bigg)\eps
      \end{align}
      holds. In other words, there exists a positive vector \(\eps\) such \(0 \le \conmat \big( \hypar\big) \eps \le 1 - \frac{\stsz}{2 \cn}\eps\) holds. We now appeal to \cite[Corollary \(8.1.29\)]{ref:RAJ-CRJ-13} to conclude \(\spec \big( \conmat(\hypar)\big)  \le 1 - \frac{\stsz}{2 \cn} \). This completes the proof.

\end{document}